\documentclass[11pt]{amsart}

\usepackage{amsmath}
\usepackage{amsxtra}
\usepackage{amscd}
\usepackage{amsthm}
\usepackage{amsfonts}
\usepackage{amssymb}
\usepackage{mathtools}
\usepackage{eucal}
\usepackage[all]{xy}
\usepackage{graphicx}
\usepackage{tikz}
\usetikzlibrary{shapes.misc}
\usepackage{tikz-cd}
\usepackage{mathrsfs}
\usepackage{euler}
\usepackage{color}
\usepackage{caption}
\usepackage{enumerate}
\usepackage{faktor}
\usepackage{bbm}
\usepackage{stmaryrd}
\usepackage[T1]{fontenc} 
\usepackage[utf8]{inputenc} 
\usepackage{tikz}
\usepackage{tikz-cd}
\usetikzlibrary{matrix, arrows, decorations.pathmorphing}
\usepackage{hyperref}
\usepackage{float}
\usepackage{multicol}
\usepackage{scrextend}

\usepackage[bottom=1.25in, top=1.5in, right=1.5in, left=1.5in]{geometry}

\theoremstyle{plain}
  \newtheorem{theorem}{Theorem}
  \newtheorem{proposition}[theorem]{Proposition}
  \newtheorem{lemma}[theorem]{Lemma}
  
  \newtheorem{maintheorem}{Theorem}
  
  \newtheorem{maincorollary}[maintheorem]{Corollary}

\theoremstyle{definition}
  \newtheorem{definition}[theorem]{Definition}
  \newtheorem{example}[theorem]{Example}

\theoremstyle{remark}
  \newtheorem{remark}[theorem]{Remark}
  
  \newtheorem*{remark*}{Remark}

\DeclareMathAlphabet{\mathcal}{OMS}{cmsy}{m}{n}

\newcommand{\A}{\mathbb{A}}

\newcommand{\C}{\mathbb{C}}

\newcommand{\N}{\mathbb{N}}

\newcommand{\Q}{\mathbb{Q}}

\newcommand{\Z}{\mathbb{Z}}

\newcommand{\cala}{\mathcal{A}}

\newcommand{\cali}{\mathcal{I}}

\newcommand{\calo}{\mathcal{O}}

\newcommand{\spec}{\mathrm{Spec}}

\newcommand{\Hom}{{\mathrm{Hom}}}

\newcommand{\SL}{{\mathrm{SL}}}

\newcommand{\paren}[1]{\mathopen{}\left(#1\right)\mathclose{}}
\newcommand{\set}[1]{\mathopen{}\left\{#1\right\}\mathclose{}}

\newcommand\restr[2]{{
  \left.\kern-\nulldelimiterspace 
  #1 
  \right|_{#2} 
  }}
\newcommand{\Mid}{\,\middle|\,}

\DeclareRobustCommand{\okina}{%
  \raisebox{\dimexpr\fontcharht\font`A-\height}{%
    \scalebox{0.8}{`}%
  }%
}

\makeatletter
\AtBeginDocument
 {
    \def\@thm#1#2#3{%
      \ifhmode
        \unskip\unskip\par
      \fi
      \normalfont
      \trivlist
      \let\thmheadnl\relax
      \let\thm@swap\@gobble
      \let\thm@indent\indent 
      \thm@headfont{\scshape}
      \thm@notefont{\fontseries\mddefault\upshape}%
      \thm@headpunct{.}
      \thm@headsep 5\p@ plus\p@ minus\p@\relax
      \thm@space@setup
      #1
      \@topsep \thm@preskip               
      \@topsepadd \thm@postskip           
      \def\dth@counter{#2}%
      \ifx\@empty\dth@counter
        \def\@tempa{%
          \@oparg{\@begintheorem{#3}{}}[]%
        }%
      \else
        \H@refstepcounter{#2}%
        \hyper@makecurrent{#2}%
        \let\Hy@dth@currentHref\@currentHref
        \AddToHookNext{para/begin}{\MakeLinkTarget*{\Hy@dth@currentHref}}%
        \def\@tempa{%
          \@oparg{\@begintheorem{#3}{\csname the#2\endcsname}}[]%
        }%
      \fi
      \@tempa
    }%
  \dth@everypar={%
    \@minipagefalse \global\@newlistfalse
    \@noparitemfalse
    \if@inlabel
      \global\@inlabelfalse
      \begingroup \setbox\z@\lastbox
       \ifvoid\z@ \kern-\itemindent \fi
      \endgroup
      \unhbox\@labels
    \fi
    \if@nobreak \@nobreakfalse \clubpenalty\@M
    \else \clubpenalty\@clubpenalty \everypar{}%
    \fi
  }%
}

\newcommand{\subalign}[1]{%
  \vcenter{%
    \Let@ \restore@math@cr \default@tag
    \baselineskip\fontdimen10 \scriptfont\tw@
    \advance\baselineskip\fontdimen12 \scriptfont\tw@
    \lineskip\thr@@\fontdimen8 \scriptfont\thr@@
    \lineskiplimit\lineskip
    \ialign{\hfil$\m@th\scriptstyle##$&$\m@th\scriptstyle{}##$\hfil\crcr
      #1\crcr
    }%
  }%
}

\makeatother

\title{Positive integral points on acyclic cluster varieties}

\author[Robin Zhang]{Robin Zhang}
\address{Massachusetts Institute of Technology\\Institut des Hautes \'{E}tudes Scientifiques}
\email{rzhang@ihes.fr}

\date{September 28, 2026}

\begin{document}

\begin{abstract}
  We determine the number of positive integral points on $n$-dimensional affine varieties associated to $n \times n$ generalized Cartan matrices. An application to the theory of cluster algebras and combinatorics is the resolution of the Fontaine--Plamondon conjecture, which says that there are exactly $4400$ and $26952$ positive integral friezes of type $E_7$ and $E_8$ respectively. An application to number theory is a refinement and generalization of theorems of Mordell, Mohanty, and Schinzel to positive integers and higher dimensions by exhibiting examples of log K3 surfaces $xyz = G(x, y)$ and affine threefolds $xyzw = G(x, y, z)$ of every degree greater than $3$ with infinitely many positive integral solutions.
\end{abstract}

\maketitle

\setcounter{tocdepth}{1}
\tableofcontents


\section{Introduction}
In this paper,
we find an unexpected connection between
Diophantine geometry and cluster algebras.
In general, this connection gives
a method to effectively bound the number of positive integral points
on affine acyclic cluster varieties of finite type,
a method to produce infinitely many positive integral points
on affine acyclic cluster varieties of infinite type,
and a method to explicitly count frieze patterns.
To demonstrate these ideas concretely,
we apply these methods to a Diophantine finiteness problem
for integral points on certain log K3 surfaces posed by Mordell \cite{mordell-1952}
and to a combinatorial enumeration problem
for Dynkin friezes posed by Fontaine and Plamondon \cite{fontaine-plamondon}.

\subsection*{Counting positive integral points}
General finiteness problems about integral and rational points on affine varieties
date back to at least 1901, when Poincar\'{e} \cite{poincare} asked whether
the group $E(\Q)$ of $\Q$-rational points on each elliptic curve $E$
has only finitely many generators.
Poincar\'{e}'s problem was affirmatively
resolved by Mordell \cite{mordell-1922} in 1922,
generalized by Weil's thesis \cite{weil-1929} in 1929 to number fields,
and generalized by Faltings \cite{faltings-1983,faltings-1983-erratum} in 1983
to curves of higher genus.
\begin{theorem}[Mordell--Weil and Faltings]
\label{thm:mordell-weil-faltings}
  Let $K$ be a number field with ring of integers $\calo_K$
  and let $X$ be a smooth projective curve of genus $g$ defined over $K$.
  \begin{enumerate}[(a)]
    \item If $g = 0$, then $\#X(K) = 0$ or $\infty$.
    \item If $g = 1$, then $X(K)$ is empty or a finitely-generated abelian group.
    \item If $g \geq 2$, then $X(K)$ is finite.
  \end{enumerate}
\end{theorem} 
Long before Faltings's proof of Theorem \ref{thm:mordell-weil-faltings}(c),
Siegel \cite{siegel}\footnote{For an English translation, see \cite{siegel-translated}.}
used Theorem \ref{thm:mordell-weil-faltings}(b)
and methods from the theory of Diophantine approximations
to prove an integral version of Theorem \ref{thm:mordell-weil-faltings}(c)
in 1929.
\begin{theorem}[Siegel's theorem on integral points]
  \label{thm:siegel}
  Let $K$ be a number field with ring of integers $\calo_K$
  and let $X$ be a smooth affine curve of genus $g$ defined over $K$.
    If $g \geq 1$, then $X(\calo_K)$ is finite.
\end{theorem}

The theorems of Siegel and Faltings were
first posed as open questions at the end of
Mordell's 1922 paper \cite{mordell-1922}.
In this article, we consider
Mordell's question over the positive integers $\N$
for a class of affine varieties of dimension $\geq 2$.

One application of our results
is a positive refinement and higher-dimensional generalization
of results in the Mordell--Schinzel program,
which looks for infinitely many integer solutions
to Diophantine equations of the form
\begin{equation}
  \label{eq:mordell-schinzel}
  xyz = G(x, y)
\end{equation}
for a polynomial $G \in \Z[x, y]$.
After proving that Equation \ref{eq:mordell-schinzel} has
infinitely many integer solutions when $G(x, y) = ax^3 + by^3 + c$,
Mordell \cite{mordell-1952} claimed (without proof)
that Equation \ref{eq:mordell-schinzel} also
has infinitely many integer solutions for any $G \in \Z[x, y]$;
quadratic counterexamples to this general statement were
found by \cite{jacobsthal,barnes,mills-1954}.
Mordell's claim was recovered in several cases by
\cite{schinzel-2015,schinzel-2018,kollar-li},
including when $G(x, y) = x^a + y + 1$ in \cite[Theorem 3]{schinzel-2015}.
We give a refinement of Schinzel's result
by showing that there are infinitely many
\textit{positive} integer solutions when $G(x, y) = x^a + y + 1$
and also prove a new higher-dimensional generalization.
\begin{maincorollary}[A Diophantine application of Theorem {\ref{thm:positive-siegel}}]
  \label{cor:rank-2-3}
  Let $a$ and $b$ be positive integers.
  \begin{enumerate}
    \item If $a \geq 4$,
    then there are infinitely many positive integer solutions
    $(x, y, z)$ to the equation
    \begin{align*}
      x y z &= x^a + y + 1.
    \end{align*}
    Furthermore if $a = 1$, $2$ or $3$,
    then the number of positive integer solutions
    is $5$, $6$, or $9$ respectively.
  \item If $ab \geq 3$,
    then there are infinitely many positive integer solutions
    $(x, y, z, w)$ to the equation
    \begin{align*}
      x y z w &= x^a y + x^b z + y + z.
    \end{align*}
  \end{enumerate}
\end{maincorollary}
\begin{remark}
  The $a \leq 3$ part of Corollary \ref{cor:rank-2-3}(1) is already known.
  Mohanty \cite[Theorem 2]{mohanty-1977} proved that
  $xyz = x^3 + y + 1$ has exactly $9$ positive integer solutions.
  We give a new proof of Mohanty's theorem
  when treating the $a = 3$ case.
  Over $\Q$, the surfaces in Corollary \ref{cor:rank-2-3}(1)
  are log K3 surfaces, as follows from
  \cite[Proposition 2.1]{blanc-dolgachev}
  and \cite[Corollary 3.28]{harpaz-log-k3}.
  The log Calabi--Yau geometry of the threefolds
  in Corollary \ref{cor:rank-2-3}(2)
  will be discussed in forthcoming joint work with de Saint Germain
  \cite{dsg-zhang}.
\end{remark}

The proof of Corollary \ref{cor:rank-2-3}
is an application (in Section \ref{sec:positive-mordell-schinzel})
of a finiteness criterion for positive integral points
on certain affine varieties with cluster algebra structures.
For the rank-two varieties in part (1),
the distinction between finite and infinite type
is precisely the distinction between $a \leq 3$ and $a \geq 4$;
this is why we have \textit{more} positive integral points
as degree grows,
which may be otherwise counter-intuitive
to how algebraic varieties of higher degree and higher genus
tend to have fewer integral and rational points.

In this article,
we study the arithmetic of
affine acyclic cluster varieties
with trivial coefficients.
For an $n \times n$ symmetrizable
generalized Cartan matrix $C = (c_{i,j})$,
let $Y_1(C)$ be the affine scheme over $\Z$
defined by the following polynomials for each $i \in \{1, \ldots, n\}$
(see Section \ref{sec:frieze-points}
and Definition \ref{def:cluster-algebra-type}):
\[
  f_{C, i}
    := x_i y_i
      - \prod_{j = 1}^{i - 1} x_j^{-c_{j, i}}
      - \prod_{j = i + 1}^n x_j^{-c_{j, i}}
      \in \Z[x_1, \ldots, x_n, y_1, \ldots, y_n].
\]
We call $Y_1(C)$ the \textit{affine frieze modular variety}
associated to $C$, and write $Y_1(\Delta_n)$ when
$C = C_{\Delta_n}$ is the generalized Cartan matrix in \nameref{sec:table-A}.
Its positive integral points parametrize positive integral friezes
by a result of de Saint Germain--Huang--Lu \cite{dhl}
(Proposition \ref{prop:correspondence}),
so we adopt notation suggested by classical and dynatomic modular curves.
In particular, the hypersurfaces of Corollary \ref{cor:rank-2-3}
are obtained from frieze modular varieties $Y_1(C)$
with generalized Cartan matrices
$\begin{psmallmatrix}
 2 & -1 \\
 -a & 2
\end{psmallmatrix}$
and
$\begin{psmallmatrix}
 2 & -1 & 0 \\
 -a & 2 & -b \\
 0 & -1 & 2
\end{psmallmatrix}$.
The resulting maps are injective on positive integral points;
in part (1), the map is an isomorphism
(see Section \ref{sec:positive-mordell-schinzel}).

The main result in this paper is
the complete determination of
positive integral points on
the frieze modular varieties $Y_1(C)$
in terms of the invariant
\[
  t_C := \min_{\cali \subset \{1, \ldots, n\}} C_\cali
    = \text{the smallest principal minor of } C,
\]
where $C_\cali$ is the determinant of the
submatrix of $C$ obtained by removing the $i$-th
column and $i$-th row of $C$ for all $i \in \cali$.
Recall that a generalized Cartan matrix $C$ is finite type
if and only if $t_C \geq 1$.

\begin{maintheorem}
  \label{thm:positive-siegel}
  Let $n$ be a positive integer,
  $C$ be an $n \times n$ symmetrizable generalized Cartan matrix,
  and $Y_1(C)$ be the corresponding frieze modular variety.
  \begin{enumerate}
    \item If $t_C \leq 0$, then $\#Y_1(C)(\N) = \infty$.
    \item If $t_C \geq 1$, then $\#Y_1(C)(\N)$ is finite
      and precisely given by
      the formulae in \nameref{sec:table-A},
      taking the product over the indecomposable components of $C$.
  \end{enumerate}
\end{maintheorem}
\begin{remark}
  We give an effective version of Theorem \ref{thm:positive-siegel}(2)
  in Proposition \ref{prop:explicit-bounds},
  which specifies bounds on the heights of $Y_1(C)(\N)$
  that are sharp for all finite types except $B_n$ and $D_n$
  and explicit for all finite types.
  The smooth locus of $Y_1(C)$ over $\Q$ is log Calabi--Yau,
  as are other smooth cluster varieties with trivial coefficients
  as shown by Gross--Hacking--Keel \cite{gross-hacking-keel}.
\end{remark}

In Section \ref{sec:positive-siegel},
we prove Theorem \ref{thm:positive-siegel}
in two steps: the qualitative finiteness/infinitude
and the quantitative enumeration.
The key cluster algebra ingredients of the
proof are:
\begin{itemize}
  \item the classification of cluster algebras
    by Fomin--Zelevinsky \cite{fomin-zelevinsky-2};
  \item the lower-bound model of
    Berenstein--Fomin--Zelevinsky \cite{fomin-zelevinsky-3};
  \item the finiteness of Dynkin friezes by
    Gunawan--Muller \cite{gunawan-muller}
    and Muller \cite{muller}; and
  \item the correspondence between positive integral friezes and
    positive integral points by de Saint Germain--Huang--Lu \cite{dhl}
    (Proposition \ref{prop:correspondence}).
\end{itemize}
The key arithmetic geometry ingredients of the proof are:
\begin{itemize}
  \item the study of $(\Z/P_{C}\Z)$-orbits on $Y_1(C)(\N)$,
    where $P_{C}$ is a positive integer associated
    to each finite type $C$
    (as listed in \nameref{sec:table-A}); and
  \item a reduction procedure of taking
    intersections with certain affine hyperplanes
    (which is shown to be equivalent to removing degree-$1$ nodes
    from the corresponding Dynkin diagram
    by Proposition \ref{prop:hyperplane}).
\end{itemize}

\noindent
In Theorem \ref{thm:positive-siegel}(2),
our counts of $\#Y_1(C)(\N)$ for types
$A_n$, $B_n$, $C_n$, $D_n$, $E_6$, $F_4$, and $G_2$
follow from direct applications of
the de Saint Germain--Huang--Lu correspondence
to the Dynkin frieze enumerations
of \cite{coxeter-conway-2,morier-genoud-ovsienko-tabachnikov,
fontaine-plamondon,bfgst}.
We also give a \textit{uniform} procedure for all types
by reducing the enumeration of $Y_1(C)(\N)$
to finite computational searches;
we perform this computation explicitly for the $E_7$ and $E_8$ cases,
for which the enumerations were not previously known.
In \hyperref[sec:appendix]{Appendix A},
we provide direct elementary proofs for finite types
$A_1$, $A_2$, $A_3$, $B_2 = C_2$, and $G_2$
because even the small rank finite types are interesting
from an arithmetic point of view: the corresponding
affine varieties have
infinitely many \textit{integral} points but
only finitely many \textit{positive integral} points
when $n \geq 2$.

\subsection*{Counting positive integral friezes}
Using Theorem \ref{thm:positive-siegel}(2),
we give an application of Diophantine geometry
to combinatorics and the theory of cluster algebras.
In particular, we look at the enumeration of friezes
when $\cala$ is of finite type $\Delta_n$.

In 1971, Coxeter \cite{coxeter}
introduced frieze patterns for each positive integer $n \in \N$:
arrays of numbers $x_{i, j}$ with $n + 2$ rows and infinitely many columns of the shape
\[
    \begin{tikzcd}[row sep=.015in, column sep = .1in, font={\small}]
        \ldots & & 1 & & 1 & & 1 & & 1 & & \ldots \\
        & \ldots & & x_{1, 1} & & x_{1, 2} & & x_{1, 3} & & \ldots & & \\
        \ldots & & x_{2, 0} & & x_{2, 1} & & x_{2, 2} & & x_{2, 3} & & \ldots \\
        & \ldots & & x_{3, 0} & & x_{3, 1} & & x_{3, 2} & & \ldots & & \\
        \ldots & & x_{4, -1} & & x_{4, 0} & & x_{4, 1} & & x_{4, 2} & & \ldots \\
        & \ddots & & \ddots & & \ddots & & \ddots & & \ddots & \\
        \\
        \ldots & & x_{n, -2} & & x_{n, -1} & & x_{n, 0} & & x_{n, 1} & & \ldots & \\
        & \ldots & & 1 & & 1 & & 1 & & \ldots & & 
    \end{tikzcd}
\]
such that all diamonds
\[
    \begin{tikzcd}[row sep=.015in, column sep = .05in, font={\small}]
        & b & \\
        a & & d \\
        & c &
    \end{tikzcd}
\]
satisfy the unimodular relation $ad - bc = 1$.
These generalize Gauss's formulae for the \textit{pentagramma mirificum},
which correspond to the $n=2$ case.
Conway--Coxeter \cite{coxeter-conway-1, coxeter-conway-2}
proved that all frieze patterns are periodic and
that there are only finitely many frieze patterns with
positive integer entries for each $n \in \N$.
They furthermore constructed a bijection with
the number of triangulations of $(n+3)$-sided polygons
to show that the number of friezes is the $(n+1)$-st Catalan number
$\frac{1}{n+2}{\binom{2n+2}{n+1}}$.

Frieze patterns were generalized
to all finite Dynkin types $\Delta_n$
by Caldero--Chapoton \cite{caldero-chapoton}
and Assem--Reutenauer--Smith \cite{assem-reutenauer-smith},
who constructed Dynkin friezes
as special $\Z$-valued homomorphisms of cluster algebras
and showed that there is a finite period
$P_{\Delta_n}$ associated to each finite type.
The Conway--Coxeter theorem corresponds to the enumeration of $A_n$-friezes
and has been extended to all Dynkin types except $E_7$ and $E_8$
by \cite{coxeter-conway-2,morier-genoud-ovsienko-tabachnikov,
fontaine-plamondon,bfgst}
using triangulations of
Euclidean $n$-gons, Poisson geometry,
triangulations of once-punctured $n$-gons,
and foldings of Dynkin diagrams.

The main combinatorial result in this article
is the resolution of the conjecture of
Fontaine--Plamondon \cite[Conjecture 4.5]{fontaine-plamondon},
which states that there are $4400$ and $26952$
many positive integral friezes of Dynkin type $E_7$ and $E_8$ respectively.
This completes the enumeration of all positive integral friezes of Dynkin type.
\begin{maincorollary}[A frieze application of Theorem {\ref{thm:positive-siegel}(2)}]
\label{cor:En}
  For each finite Dynkin type $\Delta_n$,
  the number of positive integral friezes of type $\Delta_n$
  is precisely given by the formula in \nameref{sec:table-A}.
  In particular, there are exactly $4400$ and $26952$ positive integral friezes of
  Dynkin type $E_7$ and $E_8$ respectively.
\end{maincorollary}

\begin{remark}
  Corollary \ref{cor:En} follows from the enumerative part of
  Theorem \ref{thm:positive-siegel}(2)
  and a bijection between $Y_1(C)(\N)$
  and positive integral friezes
  given by de Saint Germain--Huang--Lu \cite{dhl}
  (Proposition \ref{prop:correspondence}).
  This gives a uniform Diophantine interpretation of the enumerations in
  \cite{coxeter-conway-2,morier-genoud-ovsienko-tabachnikov,
  fontaine-plamondon,bfgst}.
\end{remark}

The study of friezes over $\Z$, $\Z[\sqrt{d}]$, and $\Z/N\Z$
has been fruitfully initiated in
\cite{fontaine-2014,cuntz-holm,cuntz-holm-jorgensen,gunawan-schiffler,
holm-jorgensen,morier-genoud-2021, bftwz, cuntz-holm-pagano,short-vanson-zabolotskii}.
One may generalize our Diophantine approach
by studying $Y_1(C)(R)$ for these
and other commutative rings $R$, both with or without
positivity conditions.
We plan to give new counts of Dynkin friezes over finite fields
in a forthcoming work.
In another direction,
friezes with general $\SL(k)$ unimodular conditions
were introduced by
Cordes--Roselle \cite{cordes-roselle};
these were shown by
Baur--Faber--Gratz--Serhiyenko--Todorov \cite{bfgst}
to arise as specializations of cluster algebras
to the positive integers just like the classical
$\SL(2)$-friezes;
so one might also use the corresponding variety to study
general $\SL(k)$-friezes.


\subsection*{Acknowledgements}

The author is deeply thankful to Antoine de Saint Germain
for kindly introducing the author to friezes
and the contents of his thesis during the
2022 AMSI--MSRI Winter--Summer School on Representation Theory
at the University of Hawai{\okina}i at Hilo,
as well as for answering the author's many na\"{i}ve follow-up questions.
The author is grateful to Noam Elkies for fruitful exchanges
on the enumerations for small $n$ and
on the initial numerical computations for type $E_n$;
to Jennifer Balakrishnan, Jane Shi,
Drew Sutherland, and John Voight for insightful discussions
about computational algorithms;
to Pietro Corvaja and Umberto Zannier
for helpful conversations about integral values of
rational functions and the work of Koll\'{a}r--Li;
and to Min Huang, Thomas Lam, and Greg Muller
for helpful comments about this paper.

The author was supported by the National Science Foundation
under Grant No. DMS-2303280.
The author acknowledges the MIT Office of Research Computing and Data
for providing performance computing resources that have contributed to
the research results reported in this article.


\numberwithin{equation}{section}
\numberwithin{theorem}{section}

\section{Affine frieze modular varieties}
\label{sec:frieze-points}

\subsection{Positive integral friezes}
\label{subsec:friezes}

By the work of Fomin--Zelevinsky \cite{fomin-zelevinsky-2},
there is a classification of cluster algebras into types
that mirrors the Killing--Cartan classification of simple Lie groups
and the classification of Dynkin diagrams.
Throughout this paper, we will
restrict ourselves to cluster algebras $\cala = \cala_C$
associated to an $n \times n$ symmetrizable
generalized Cartan matrix $C = (c_{i, j})$
with trivial coefficients and exchange matrix
\[
  B = \begin{pmatrix}
  0 & -c_{1, 2} & \ldots & -c_{1, n} \\
  c_{2, 1} & 0 & \ddots & \vdots \\
  \vdots & \ddots & \ddots & -c_{n-1, n} \\
  c_{n, 1} & \ldots & c_{n, n-1} & 0
  \end{pmatrix}.
\]
Thus the exponents in the $i$-th exchange relation
are given by the $i$-th column of $C$,
as in \cite[Section 6.3]{dhl}.
When $\cala$ is of finite type,
there is a Dynkin type $\Delta_n$ of rank $n$ 
such that when $C$ is indecomposable,
it agrees with a Cartan matrix
$C_{\Delta_n}$ of $\Delta_n$ up to a simultaneous permutation
of rows and columns
(see \nameref{sec:table-A} for the list of finite Dynkin types
$\Delta_n$ and a choice of $C_{\Delta_n}$);
we will simply use $\Delta_n$ to refer to the type of $\cala$
in this case.

Following Caldero--Chapoton \cite{caldero-chapoton}
and Assem--Reutenauer--Smith \cite{assem-reutenauer-smith},
a positive integral frieze
is a ring homomorphism from
a cluster algebra $\cala$ to $\Z$
such that the cluster variables are sent to
the set $\N$ of positive integers.
Positive integral friezes were used in the
work of Assem--Reutenauer--Smith
to recover results of
\cite{fomin-zelevinsky-1, fomin-zelevinsky-2, fomin-zelevinsky-3, fomin-zelevinsky-4}
and to prove new explicit formulae for cluster variables.
In this generalization,
the classical Coxeter friezes
are the positive integral friezes for cluster algebras
of type $A_n$.

By the work of Fomin--Zelevinsky \cite{fomin-zelevinsky-2}
(cf. \cite[Corollary 1]{assem-reutenauer-smith},
\cite[Theorem 2.7]{morier-genoud}),
if $C$ is of finite type $\Delta_n$
then each $C$-frieze $F$ (which we also call a $\Delta_n$-frieze)
is periodic under horizontal translation with some period $P_{\Delta_n}$.
In particular, each $\Delta_n$-frieze can be uniquely represented by
an $n \times P_{\Delta_n}$ grid of positive integers
\[
  \begin{tikzcd}[row sep=.015in, column sep = .1in]
    F_{1,1} & F_{1,2} & \cdots & F_{1,P_{\Delta_n}} \\
    F_{2,1} & F_{2,2} & \cdots & F_{2,P_{\Delta_n}} \\
    \vdots & \vdots & \ddots & \vdots \\
    F_{n,1} & F_{n,2} & \cdots  & F_{n,P_{\Delta_n}}
  \end{tikzcd}
\]
such that it satisfies the mesh relations for all $(i, j)$
in terms of the Cartan matrix $C = C_{\Delta_n} = (c_{i, j})$:
\[
    F_{i, j} F_{i, j + 1}
        = 1 + \prod_{k = 1}^{i - 1} F_{k, j + 1}^{-c_{k, i}}
          \prod_{k = i + 1}^n F_{k, j}^{-c_{k, i}}.
\]

\begin{example}[A frieze of type $E_8$]
    \label{ex:E8}
    \[
      \begin{tikzcd}[row sep=.025in, column sep = .1in, font = {\small}]
        \ldots & 1 & & 1 & & 1 & & 1 & & 1 & & 1 & \ldots \\
        \ldots & & 4 & & 3 & & 3 & & 4 & & 4 & & \ldots \\
        \ldots & 15 & & 11 & & 8 & & 11 & & 15 & & 11 & \ldots \\
        \ldots & 7 & 41 & 6 & 29 & 5 & 29 & 6 & 41 & 7 & 41 & 6 & \ldots \\
        \ldots & 16 & & 18 & & 21 & & 18 & & 16 & & 18 & \ldots \\
        \ldots & & 7 & & 13 & & 13 & & 7 & & 7 & & \ldots \\
        \ldots & 3 & & 5 & & 8 & & 5 & & 3 & & 5 & & \\
        \ldots & & 2 & & 3 & & 3 & & 2 & & 2 & & \ldots \\
        \ldots & 1 & & 1 & & 1 & & 1 & & 1 & & 1 & \ldots \\
      \end{tikzcd}
    \]
    is a frieze of type $E_8$. In particular, all diamonds
    \[
      \begin{tikzcd}[row sep=.015in, column sep = .05in, font = {\small}]
        & b & \\
        a & & d \\
        & c & \\
      \end{tikzcd}
      \qquad \text{ and } \qquad
      \begin{tikzcd}[row sep=.015in, column sep = .05in, font = {\small}]
        & f & \\
        e & g & i \\
        & h & \\
      \end{tikzcd}
    \]
    satisfy the unimodular relations $ad - bc = 1$
    and $ei - fgh = 1$.
    This $E_8$-frieze has exact period $4$
    and can be represented as the $8 \times 4$ grid:
    \[
      \begin{tikzcd}[row sep=.01in, column sep = .1in]
        6 & 5 & 6 & 7 \\
        4 & 3 & 3 & 4 \\
        11 & 8 & 11 & 15 \\
        29 & 29 & 41 & 41 \\
        21 & 18 & 16 & 18 \\
        13 & 7 & 7 & 13 \\
        5 & 3 & 5 & 8 \\
        2 & 2 & 3 & 3
      \end{tikzcd}
    \]
    Notice that the first entry in each column of the grid
    corresponds to a subdiagonal of the frieze
    and the rest of each column
    corresponds to the adjacent diagonal of the frieze.
    A different choice of Cartan matrix for $E_8$
    would give rise to a different presentation
    of the frieze. Our choice of Cartan matrix
    given in \nameref{sec:table-A} corresponds
    to the following labeling of the Dynkin diagram for $E_8$:
    
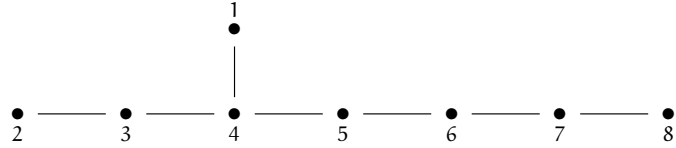
\[
      \begin{tikzcd}
          & & \overset{1}{\bullet} \arrow[d, dash] \\
          \underset{2}{\bullet} \arrow[r, dash] & \underset{3}{\bullet} \arrow[r, dash] & \underset{4}{\bullet} \arrow[r, dash] & \underset{5}{\bullet} \arrow[r, dash] & \underset{6}{\bullet} \arrow[r, dash] & \underset{7}{\bullet} \arrow[r, dash] & \underset{8}{\bullet}
      \end{tikzcd}
    \]
\end{example}

\subsection{Frieze polynomials and affine varieties}
\label{subsec:frieze-polynomials}
Let $C$ be an $n \times n$ symmetrizable
generalized Cartan matrix with entries $c_{i, j}$.
Consider the affine space $\A_\Z^{2n}$
with coordinates $x = (x_1, \ldots, x_n; y_1, \ldots, y_n)$.
The following definition is based on the
lower bound model of
Berenstein--Fomin--Zelevinsky \cite[Section 1.3]{fomin-zelevinsky-3}
and the model of Geiss--Leclerc--Schr\"{o}er
\cite[Proposition 7.4]{geiss-leclerc-schroer}
for the cluster algebra $\cala_C$ with trivial coefficients
associated to $C$.
\begin{definition}
    \label{def:frieze-polynomial}
    For each integer $i \in \{1, \ldots, n\}$,
    the $i$-th \textit{lower bound $C$-frieze polynomial} is
    \[
      f_{C, i}
        := x_i y_{i}
          - \prod_{j = 1}^{i - 1} x_j^{-c_{j, i}}
          - \prod_{j = i + 1}^n x_j^{-c_{j, i}}
          \in \Z[x_1, \ldots, x_n, y_1, \ldots, y_n],
    \]
    and the $i$-th \textit{Geiss--Leclerc--Schr\"{o}er $C$-frieze polynomial} is
    \[
      g_{C, i}
        := x_i z_{i}
          - \prod_{j = i + 1}^n x_j^{-c_{j, i}}
          \prod_{j = 1}^{i - 1} z_j^{-c_{j, i}} - 1
          \in \Z[x_1, \ldots, x_n, z_1, \ldots, z_n].
    \]
    The corresponding coordinate rings are
    $R_{C} := \Z[x_1, \ldots, x_n, y_1, \ldots, y_n]
    / (f_{C, 1}, \ldots, f_{C, n})$
    and
    $S_{C} := \Z[x_1, \ldots, x_n, z_1, \ldots, z_n]
    / (g_{C, 1}, \ldots, g_{C, n})$.
\end{definition}

These rings give two presentations of the affine acyclic cluster variety
associated to $\cala_C$.
We regard $\cala_C$ as a $\Z$-algebra.
Since $\cala_C$ is acyclic, there are isomorphisms
$\cala_C \cong R_C \cong S_C$
(see \cite[Section 6.3]{dhl}).
In this article we use the lower bound model,
whose defining equations are convenient for the eliminations
and hyperplane sections below.
\begin{definition}
  \label{def:cluster-algebra-type}
  We call $Y_1(C) := \spec(R_C)$ the
  \textit{affine frieze modular variety} associated to $C$.
  When $C = C_{\Delta_n}$, we also write $Y_1(\Delta_n)$.
\end{definition}
For the remainder of this article,
``frieze polynomial'' refers to a lower bound frieze polynomial $f_{C, i}$.

For every commutative ring $R$, we have
\[
  Y_1(C)(R) = \Hom_{\Z\text{-alg}}(R_C, R)
    \cong \Hom_{\Z\text{-alg}}(\cala_C, R).
\]
Thus $Y_1(C)$ represents specializations of the cluster algebra to $R$.
We write $Y_1(C)(\N)$ for the integral points of $Y_1(C)$
at which $x_1, \ldots, x_n, y_1, \ldots, y_n$ are all positive.

The complex fiber of $Y_1(C)$ is an $n$-dimensional affine variety,
and its complex points are given by
\[
    Y_1(C)(\C)
      = \set{x \in \A_\C^{2n} \Mid
        f_{C, 1}(x) = \ldots = f_{C, n}(x) = 0}.
\]
The complex fiber of $Y_1(C)$ is smooth at the points coming from $Y_1(C)(\N)$
but can have singularities;
see \cite{bfms}, which exhibits singularities of
$\spec(\cala_{C})$
for the cluster algebra $\cala_{C}$ of type $\Delta_n$.
Note that the variety $Y_1(C)$ is different than the frieze variety $X(Q)$
associated to an acyclic quiver $Q$
as defined by Lee--Li--Mills--Schiffler--Seceleanu \cite{llmss}.
Also note that the set of $\Delta_n$-frieze polynomials is
only unique up to the choice of Cartan matrix for $\Delta_n$,
but the number of positive integral points does not depend on
the choice of Cartan matrix $C$
by Proposition \ref{prop:correspondence}.
See \nameref{sec:table-A} for
a list of frieze polynomials of each Dynkin type.

A key ingredient in our proofs is the following
correspondence between friezes
and affine varieties due to
de Saint Germain--Huang--Lu \cite[Section 6.3]{dhl}
(see \cite[Example 3.4.2]{dsg-thesis} for an illustration
in type $E_7$),
which we write in our frieze modular variety interpretation.
\begin{proposition}[{\cite[Section 6.3]{dhl}}]
  \label{prop:correspondence}
  Let $C$ be an $n \times n$ symmetrizable generalized Cartan matrix.
  There is a one-to-one correspondence between
  the set of positive integral $C$-friezes
  and $Y_1(C)(\N)$.
\end{proposition}

In the finite type case,
a consequence of Proposition \ref{prop:correspondence}
is that the positive integrality of the first column
and $n$ other entries of a Dynkin frieze $F$
are sufficient to both
uniquely determine all other entries
and guarantee their positive integrality.
A point $x \in Y_1(\Delta_n)(\N)$
corresponds to the $\Delta_n$-frieze with the first column
given by the $x_i = F_{i, 1}$
and with $n$ other entries given by the $y_i$;
here $y_i$ is the value of the cluster variable obtained
by mutating the initial seed in direction $i$.
The $y_i$ are generally not confined to any single column.

Since friezes of Dynkin type $\Delta_n$
are horizontally periodic with period $P_{\Delta_n}$,
there is an action of $\Z/P_{\Delta_n}\Z$ on
$Y_1(\Delta_n)(\N)$.
For $x \in Y_1(\Delta_n)(\N)$,
let $F$ denote the corresponding $\Delta_n$-frieze.
For each $\sigma_m \in \Z/P_{\Delta_n}\Z$,
define $\sigma_m(x)$ to be the point of $Y_1(\Delta_n)(\N)$
whose corresponding frieze has first column
$(F_{1,m+1}, \ldots, F_{n,m+1})$,
i.e. the horizontal translation of
$F$ to the right by $m$ columns.
de Saint Germain has kindly pointed out
to the author that the $\sigma_1$
action has the following names in different contexts:
\begin{itemize}
    \item the Auslander--Reiten translation
      in finite-dimensional representation theory
      (cf. \cite[Section 2.1]{llmss});
    \item the Fomin--Zelevinsky twist in Lie theory
      (cf. \cite[Section 5.6]{dsg-thesis});
    \item the Donaldson--Thomas transformation
      in Calabi--Yau theory
      (cf. \cite{weng});
    \item the maximal green sequence in combinatorics
      (cf. \cite{mills-2017}).
\end{itemize}

\begin{remark}
  An immediate corollary of the theorem of Conway--Coxeter \cite{coxeter-conway-2}
  and Proposition \ref{prop:correspondence}
  is an expression for the $n$-th Catalan number
  as the number of positive integer solutions
  to a Diophantine system of equations:
  \[
    \frac{1}{n + 1}{\binom{2n}{n}}
      = \#Y_1(A_{n - 1})(\N)
      = \# \set{x \in \N^{2n-2} \Mid
        \subalign{
        x_1 y_{1} &= x_2 + 1, \\
        x_2 y_{2} &= x_1 + x_3, \\
        x_3 y_{3} &= x_2 + x_4, \\
        & \vdots \\
        x_{n - 2} y_{n - 2} &= x_{n - 3} + x_{n - 1} \\
        x_{n - 1} y_{n - 1} &= x_{n - 2} + 1
        }}.
  \] 
\end{remark}

\subsection{Reduction to smaller types}
\label{sec:reduction}
We observe that for our affine varieties
defined in Section \ref{subsec:frieze-polynomials},
removing a node from the Dynkin diagram of type $\Delta_n$
corresponds to taking an intersection of $Y_1(\Delta_n)$
with the affine hyperplane to the removed node.
For notational convenience, we will fix a Cartan matrix
$C_{\Delta_n} = (c_{i, j})$ and label the nodes
from $1$ to $n$
of the Dynkin diagram accordingly.

\begin{proposition}
    \label{prop:hyperplane}
    Let $\Delta_n$ be a Dynkin type of rank $n$
    and let $\Delta_n^{(k)}$ be a Dynkin type
    obtained by removing the $k$-th node
    from the Dynkin diagram of $\Delta_n$.
    There is an isomorphism
    $Y_1(\Delta_n^{(k)}) \cong Y_1(\Delta_n) \cap \{x_k = 1\}$
    that preserves positive integral points;
    furthermore, if the degree of the $k$-th node in $\Delta_n$ is $1$,
    then there is an isomorphism
    $Y_1(\Delta_n^{(k)}) \cong Y_1(\Delta_n) \cap \{y_k = 1\}$
    that preserves positive integral points.
\end{proposition}

\begin{proof}
    First set $x_k = 1$.
    The $k$-th frieze equation determines $y_k$,
    and the remaining equations are precisely those obtained
    by deleting the $k$-th row and column of $C_{\Delta_n}$.
    Eliminating $y_k$ gives
    $Y_1(\Delta_n) \cap \{x_k = 1\} \cong Y_1(\Delta_n^{(k)})$.
    This preserves positive integral points.

    Now suppose that the $k$-th node has degree $1$,
    with unique neighbor $k'$.
    Write $r = -c_{k,k'}$ and $s = -c_{k',k}$.
    The two frieze equations involving $x_k$ are
    \begin{align*}
        x_k y_k &= x_{k'}^s + 1, \\
        x_{k'} y_{k'} &= x_k^r M + N,
    \end{align*}
    where $M$ and $N$ are monomials in the other $x_i$.
    Setting $y_k = 1$ gives $x_k = x_{k'}^s + 1$.
    Since
    \[
      Q(t) := \frac{(t^s + 1)^r - 1}{t} \in \Z[t],
    \]
    the second equation becomes
    \[
      x_{k'}\big(y_{k'} - M Q(x_{k'})\big) = M + N.
    \]
    Thus eliminating $x_k$ and replacing $y_{k'}$
    by $y_{k'} - M Q(x_{k'})$ gives precisely
    the frieze equations for $\Delta_n^{(k)}$.
    These polynomial changes of variables give
    $Y_1(\Delta_n) \cap \{y_k = 1\} \cong Y_1(\Delta_n^{(k)})$.
    They also preserve positive integral points:
    the displayed equation shows that
    $y_{k'} - M Q(x_{k'}) > 0$ on the original positive locus,
    while the inverse substitution adds the nonnegative quantity
    $M Q(x_{k'})$.
\end{proof}

For the alphabetic families in \nameref{sec:table-A},
the Dynkin diagram for $\Delta_{n + 1}$ is obtained from that for $\Delta_n$
by adjoining a node at the end of the chain
(with the indexing in the table, the added node is the last node).
By Proposition \ref{prop:hyperplane},
\[
  Y_1(\Delta_n) \cong Y_1(\Delta_{n + 1}) \cap \set{x_{n + 1} = 1}.
\]
For example, extending a leaf with equation
\[
  x_i y_i = x_j^r + 1
\]
replaces this equation by
\begin{align*}
  x_i y_i &= x_j^r + x_k, \\
  x_k y_k &= x_i + 1.
\end{align*}

\begin{example}
    Let $\Delta_n = E_7$. We can
    obtain the Dynkin diagram for $E_6$
    by removing the node labeled ``$7$''
    from the Dynkin diagram of $E_7$:
    
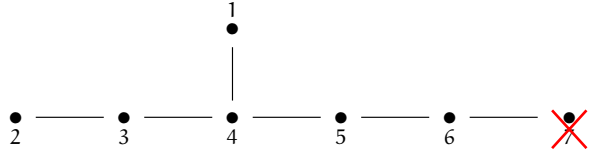
\[
        \begin{tikzcd}
            & & \overset{1}{\bullet} \arrow[d, dash] \\
            \underset{2}{\bullet} \arrow[r, dash] & \underset{3}{\bullet} \arrow[r, dash] & \underset{4}{\bullet} \arrow[r, dash] & \underset{5}{\bullet} \arrow[r, dash] & \underset{6}{\bullet} \arrow[r, dash] & \tikz{\node[cross out, draw=red, line width=1pt, inner sep=0pt, minimum size=12pt] {\(\underset{7}{\bullet}\)};}
        \end{tikzcd}
    \]
    The $E_7$-frieze equations
    for the Cartan matrix in \nameref{sec:table-A} are:
    \begin{align*}
        x_1 y_1 &= x_4 + 1, \\
        x_2 y_2 &= x_3 + 1, \\
        x_3 y_3 &= x_2 + x_4, \\
        x_4 y_4 &= x_1 x_3 + x_5, \\
        x_5 y_5 &= x_4 + x_6, \\
        x_6 y_6 &= x_5 + x_7, \\
        x_7 y_7 &= x_6 + 1.
    \end{align*}
    Observe that the $\N^{14}$-solutions
    of the specialization of the $E_7$-frieze equations
    to $y_7 = 1$:
    \begin{align*}
        x_1 y_1 &= x_4 + 1, \\
        x_2 y_2 &= x_3 + 1, \\
        x_3 y_3 &= x_2 + x_4, \\
        x_4 y_4 &= x_1 x_3 + x_5, \\
        x_5 y_5 &= x_4 + x_6, \\
        x_6 (y_6 - 1) &= x_5 + 1, \\
    \end{align*}
    are in bijection with the $\N^{12}$-solutions
    to the $E_6$-frieze equations
    in $\Z[x_1, \ldots, x_6, y_1, \ldots, y_6']$
    (after observing that there are no positive integral solutions
    with $y_6 = 1$ and changing variables $y_6 - 1 \mapsto y_6'$).
    Similarly, the $\N^{14}$-solutions of the specialization
    of the $E_7$-frieze equations to $x_7 = 1$:
    \begin{align*}
        x_1 y_1 &= x_4 + 1, \\
        x_2 y_2 &= x_3 + 1, \\
        x_3 y_3 &= x_2 + x_4, \\
        x_4 y_4 &= x_1 x_3 + x_5, \\
        x_5 y_5 &= x_4 + x_6, \\
        x_6 y_6 &= x_5 + 1, \\
    \end{align*}
    are in bijection with the $\N^{12}$-solutions
    to the $E_6$-frieze equations
    in $\Z[x_1, \ldots, x_6, y_1, \ldots, y_6]$.
\end{example}

If an entry of a positive integral $\Delta_n$-frieze is $1$,
then a horizontal translation gives a point of
$Y_1(\Delta_n)(\N) \cap \{x_k = 1\}$ for some $k$.
By Proposition \ref{prop:hyperplane}, this point corresponds
to a positive integral point of a Dynkin type of smaller rank.
Thus it remains to consider friezes whose entries are all at least $2$.
For $E_7$ and $E_8$, this recovers
the observation of \cite[Remark 6.21]{gunawan-muller}
that it is sufficient, given the enumerations of smaller ranks,
to search for friezes
whose entries are at least $2$.
In fact, their observation uses
considerations of cluster algebras and generalized associahedra
after removing vertices from Dynkin diagrams
in \cite[Appendix A]{gunawan-muller}.

We translate the observation of \cite[Remark 6.21]{gunawan-muller}
into a statement about $Y_1(\Delta_n)$,
which will be useful in Section \ref{sec:En}
when proving the $E_8$ enumeration theorem.
Recall from Section \ref{subsec:frieze-polynomials}
that there is an action of $\Z/P_{\Delta_n}\Z$
on $Y_1(\Delta_n)(\N)$ defined via horizontal translations $\sigma_m$
of the corresponding friezes by $m$ columns.
\begin{lemma}
  \label{lem:friezes-at-least-two}
  Let $\calo_{\Z/P_{\Delta_n}\Z}(x)
  := \set{x, \sigma_1(x), \ldots, \sigma_{P_{\Delta_n} - 1}(x)}$
  denote the $(\Z/P_{\Delta_n}\Z)$-orbit of a point $x \in Y_1(\Delta_n)(\N)$.
  \begin{enumerate}[(a)]
    \item If there are no points $x \in Y_1(E_7)(\N)$
      such that $\calo_{\Z/10\Z}(x) \subset Y_1(E_7)(\Z_{\geq 2})$,
      then $\#Y_1(E_7)(\N) = 4400$.
    \item If the hypothesis of (a) holds
      and there are exactly $4$ points $x \in Y_1(E_8)(\N)$
      such that $\calo_{\Z/16\Z}(x) \subset Y_1(E_8)(\Z_{\geq 2})$,
      then $\#Y_1(E_8)(\N) = 26952$.
  \end{enumerate}
\end{lemma}
\begin{proof}
  A point $x = (x_1, \ldots, x_n; y_1, \ldots, y_n)
  \in Y_1(\Delta_n)(\N)$ corresponds
  under Proposition \ref{prop:correspondence}
  to a frieze given by
  $F_{i, j} = (\sigma_{j-1}(x))_i$.
  Therefore, a point $x \in Y_1(\Delta_n)(\N)$ whose
  entire $(\Z/P_{\Delta_n}\Z)$-orbit
  has coordinates at least $2$
  necessarily corresponds to a $\Delta_n$-frieze
  whose entries are all at least $2$.

  As listed in \nameref{sec:table-A},
  $P_{E_7} = 10$;
  hence the assumption of (a) that there are no points $x \in Y_1(E_7)(\N)$
  such that $\calo_{\Z/10\Z}(x) \subset Y_1(E_7)(\Z_{\geq 2})$
  implies the absence of $E_7$-friezes whose entries are all at least $2$.
  Then there are exactly $4400$ friezes of type $E_7$
  by \cite[Remark 6.21]{gunawan-muller}
  (or the aforementioned application of Proposition \ref{prop:hyperplane}).
  Under the hypothesis of (a), there are exactly $26948$
  positive integral $E_8$-friezes with an entry equal to $1$
  by \cite[Appendix A.5]{gunawan-muller};
  hence the same argument with $P_{E_8} = 16$
  and the additional assumption
  that there are exactly $4$ friezes of type $E_8$
  whose entries are all at least $2$
  gives $\#Y_1(E_8)(\N) = 26948 + 4 = 26952$.
\end{proof}

\subsection{Positive integral values of rational functions}
\label{sec:rational-functions}
By rearranging the vanishing of frieze polynomials
for $y_i$,
one obtains $n$ equations
$y_i = \varphi_{C, i}$
involving rational functions $\varphi_{C, i}$.
Hence finding a positive integral point
$(x_1, \ldots, x_n; y_1, \ldots, y_{n}) \in Y_1(C)(\N)$
is equivalent to finding positive integral points
$x = (x_1, \ldots, x_n) \in \N^n$
on which $\varphi_{C, i}(x) \in \N$
simultaneously for all $i$.

\begin{example}
    \label{ex:E7-rational-functions}
    In the case $\Delta_n = E_7$ (cf. \cite[Example 3.4.2]{dsg-thesis}),
    $Y_1(E_7)(\N)$ is in bijection with the points
    $x = (x_1, \ldots, x_7) \in \N^7$ at which the following 
    $7$ rational functions simultaneously take positive integral values:
    \begin{align*}
        \varphi_{E_7, 1}(x) &:= \frac{x_4 + 1}{x_1}, & \varphi_{E_7, 5}(x) &:= \frac{x_4 + x_6}{x_5}, \\
        \varphi_{E_7, 2}(x) &:= \frac{x_3 + 1}{x_2}, & \varphi_{E_7, 6}(x) &:= \frac{x_5 + x_7}{x_6}, \\
        \varphi_{E_7, 3}(x) &:= \frac{x_2 + x_4}{x_3}, & \varphi_{E_7, 7}(x) &:= \frac{x_6 + 1}{x_7}. \\
        \varphi_{E_7, 4}(x) &:= \frac{x_1 x_3 + x_5}{x_4}, & &
    \end{align*}
    By Theorem \ref{thm:positive-siegel}(2), there are exactly $4400$
    $7$-tuples of positive integers $x$ such that
    these rational functions $\varphi_{E_7, i}$ take values in the positive integers.
\end{example}

Therefore, the determination of points in $Y_1(C)(\N)$
can be viewed as a variant of another famous problem in number theory
with \textit{positivity}.
Over the integers,
the general question of determining when there are points
$x \in \Z^n$ for two polynomials $p, q \in \Z[x_1, \ldots, x_n]$
such that $p(x)$ divides $q(x)$ is a fascinating open problem
and is closely related to other famous problems such
as the dynamical Mordell--Lang conjecture.
Some finiteness and Zariski density criteria for $S$-integral points
are known in the $1$- and $2$-variable cases
over any number field
due to Siegel \cite{siegel,siegel-translated} and Corvaja--Zannier
\cite{corvaja-zannier-2004, corvaja-zannier-2010} respectively.
In the $1$-variable case, a genus $0$ case of Siegel's theorem
for a ring of $S$-integers $R$ gives the following statement:
if $p(X), q(X) \in R[X]$ are coprime polynomials such that
$p(X)$ has at least two distinct complex roots, then there are only
finitely many $x \in R$ such that $p(x)$ divides $q(x)$ in $R$;
the works of Corvaja--Zannier extend Siegel's theorem to the $2$-variable setting.
The general $n$-variable case is related to the Lang--Vojta conjecture
and the broader conjectures of Vojta
(see
\cite[Conjecture F.5.3.2, Conjecture F.5.3.6, Conjecture F.5.3.8]{hindry-silverman},
\cite[Section 1]{corvaja-zannier-2010},
\cite[Section 2]{zhang-2024}).
A forthcoming work of Corvaja--Zannier \cite{corvaja-zannier-new}
provides interesting discussions about rational functions that take
very few integral values, many of which are closely related
to those obtained from frieze polynomials
(also see Remark \ref{rem:G2}).


\section{Finiteness and enumeration of positive integral points}
\label{sec:positive-siegel}

In this section, we prove the main theorem about integral points.
In Section \ref{sec:finiteness-infinitude},
we prove the qualitative parts of Theorem \ref{thm:positive-siegel};
in Section \ref{sec:En},
we prove the enumerative parts of Theorem \ref{thm:positive-siegel}(2).

\subsection{Finiteness and infinitude}
\label{sec:finiteness-infinitude}
First, we combine existing cluster algebra theorems
with our interpretation of frieze modular varieties
to prove Theorem \ref{thm:positive-siegel}(1)
and the finiteness in \ref{thm:positive-siegel}(2).

Let $C$ be an $n \times n$ symmetrizable generalized Cartan matrix
and consider the frieze modular variety $Y_1(C)$
from Definition \ref{def:cluster-algebra-type}.
By the classification of cluster algebras
due to Fomin--Zelevinsky \cite{fomin-zelevinsky-2},
$\cala = \cala_C$ has only finitely many cluster variables if and only if
$C$ is of finite type.
As observed by Morier-Genoud \cite[Section 2.8]{morier-genoud-2012}
(see also \cite[Proposition 4.11]{morier-genoud}
and \cite[Proposition 3.17]{gunawan-muller}),
specializing all variables in a cluster to $1$
gives a positive integral frieze,
and distinct clusters give distinct friezes.
Hence if $C$ is of infinite type,
then there are infinitely many positive integral $C$-friezes.
Gunawan--Muller \cite[Theorem B]{gunawan-muller}
proved the converse: if $C$ is of finite type
then there are only finitely many positive integral $C$-friezes.
Therefore we have shown that there are
only finitely many positive integral $C$-friezes
if and only if $C$ is of finite type.

By Proposition \ref{prop:correspondence},
there are only finitely many positive integral points on $Y_1(C)$
if and only if $C$ is of finite type.
Since $C$ is of finite type precisely when $t_C > 0$,
we can conclude that $\#Y_1(C)(\N) < \infty$
if and only if $t_C > 0$.

\subsection{Enumeration}
\label{sec:En}
We now prove the enumerative part of Theorem \ref{thm:positive-siegel}(2).
For finite types $\Delta_n$,
the one-to-one correspondence of Proposition \ref{prop:correspondence}
implies that $\#Y_1(\Delta_n)(\N)$ equals the number of positive integral friezes
of type $\Delta_n$. For each $\Delta_n \notin \{E_7, E_8\}$,
the number of positive integral friezes of type $\Delta_n$
was enumerated in
\cite{coxeter-conway-2,morier-genoud-ovsienko-tabachnikov,
fontaine-plamondon,bfgst}.
We complete the enumeration by determining $\#Y_1(E_7)(\N)$ and $\#Y_1(E_8)(\N)$.

We directly prove that $\#Y_1(E_8)(\N) = 26952$ and then deduce that
$\#Y_1(E_7)(\N) = 4400$ from the $E_8$ case
using Proposition \ref{prop:hyperplane}.
Note that the method of proof for $E_8$ also works
with slight modification for $E_7$.


\subsubsection*{Type \texorpdfstring{$E_8$}{E8}}

Using the structure of
the affine variety $Y_1(E_8)$,
we give explicit bounds on its
positive integral points
to prove the $E_8$ case of Theorem \ref{thm:positive-siegel}(2).

The $E_8$-frieze equations are:
\begin{align*}
  x_1 y_1 &= x_4 + 1 \\
  x_2 y_2 &= x_3 + 1 \\
  x_3 y_3 &= x_2 + x_4 \\
  x_4 y_4 &= x_1 x_3 + x_5 \\
  x_5 y_5 &= x_4 + x_6 \\
  x_6 y_6 &= x_5 + x_7 \\
  x_7 y_7 &= x_6 + x_8 \\
  x_8 y_8 &= x_7 + 1
\end{align*}

First, we observe that all of the coordinates
of a positive integral solution are
essentially bounded from above by $x_4$.
\begin{lemma}
    \label{lem:E8-x4-largest}
    Let $(x_1, \ldots, x_8; y_1, \ldots y_8) \in Y_1(E_8)(\N)$.
    If $x_4, y_3, y_5, y_6, y_7 \geq 2$, then $x_i, y_i \leq x_4 + 4$
    for all $i \in \{1, \ldots, 8\}$.
\end{lemma}

\begin{proof}
    The bounds for $x_1$ and $y_1$ follow 
    immediately from the equation $x_1 y_1 = x_4 + 1$:
    we have that
    \begin{itemize}
        \item $x_1, y_1 \leq x_4 + 1$.
    \end{itemize}
    
    From the equation $x_2 y_2 = x_3 + 1$,
    we have that $x_2 \leq x_3 + 1$.
    Substituting this into the third equation,
    we have that $x_3 (y_3 - 1) \leq x_4 + 1$. But $y_3 > 1$,
    so:
    \begin{itemize}
        \item $x_3 \leq x_4 + 1$ and $y_3 \leq x_4 + 2$;
        \item $x_2, y_2 \leq x_4 + 2$.
    \end{itemize}
    
    From substituting $x_8 \leq x_7 + 1$ into the seventh equation,
    we have that $x_7 (y_7 - 1) \leq x_6 + 1$. But $y_7 > 1$,
    so $x_7 \leq x_6 + 1$.
    Then applying $x_7 \leq x_6 + 1$ to the sixth equation,
    we have that $x_6 (y_6 - 1) \leq x_5 + 1$. But $y_6 > 1$,
    so $x_6 \leq x_5 + 1$. Applying this to the fifth equation,
    we have that $x_5 (y_5 - 1) \leq x_4 + 1$. But $y_5 > 1$,
    so $x_5 \leq x_4 + 1$. Hence:
    \begin{itemize}
        \item $x_5 \leq x_4 + 1$ and $y_5 \leq x_4 + 2$;
        \item $x_6 \leq x_4 + 2$ and $y_6 \leq x_4 + 3$;
        \item $x_7 \leq x_4 + 3$ and $y_7 \leq x_4 + 4$;
        \item $x_8, y_8 \leq x_4 + 4$.
    \end{itemize}
    Finally by the previous bounds,
    $y_4 = \frac{x_1 x_3 + x_5}{x_4}
    \leq \frac{x_4^2 + 3x_4 + 2}{x_4}$.
    Since $x_4 \geq 2$, this gives $y_4 \leq x_4 + 4$.
\end{proof}

\begin{remark}
    For computational purposes, the bounds in
    Lemma \ref{lem:E8-x4-largest} can be refined further.
    When searching for friezes with entries all
    at least $2$, i.e. $x_i, y_i \geq 2$ for all $i$,
    we have the following bounds by similar arguments:
    \begin{multicols}{2}
        \begin{itemize}
            \item $x_1, y_1 \leq \frac{x_4 + 1}{2}$,
            \item $x_2, y_2 \leq \frac{x_4 + 2}{3}$,
            \item $x_3, y_3 \leq \frac{2x_4 + 1}{3}$,
            \item $y_4 \leq \frac{x_4}{3} + 1$,
            \item $x_5, y_5 \leq \frac{4x_4 + 1}{5}$,
            \item $x_6, y_6 \leq \frac{3x_4 + 2}{5}$,
            \item $x_7, y_7 \leq \frac{2x_4 + 3}{5}$,
            \item $x_8, y_8 \leq \frac{x_4 + 4}{5}$.
        \end{itemize}
    \end{multicols}
\end{remark}

Muller \cite[Proposition 2.3 and Example 3.1]{muller}
gives an upper bound on each entry of a frieze.
Since $x_4$ equals $F_{4, 1}$ on the corresponding frieze,
this translates into upper bounds on $x_4$
of roughly $1.7 \cdot 10^{650}$
for a general $E_8$-frieze
and roughly $4.4 \cdot 10^{266}$
for an $E_8$-frieze whose entries are all at least $2$.
These are unfortunately far too large for a computational search on $Y_1(E_8)$,
but we can obtain an exponential improvement
from a combination of our Diophantine model
and a modification of the proof of \cite[Proposition 2.3]{muller}
with the specific mesh relations for $E_8$.

Recall that there is an action of $\Z/16\Z$
on $Y_1(E_8)(\N)$ induced by horizontal translation
of the corresponding friezes.
As in Lemma \ref{lem:friezes-at-least-two},
let $\calo_{\Z/16\Z}(x)$ denote the
$(\Z/16\Z)$-orbit of $x$ and let $\Z_{\geq 2}$
denote the integers at least $2$.
Let $\pi_i: Y_1(E_8)(\C) \rightarrow \C$
denote the projection onto the $i$-th coordinate.

\begin{lemma}
  \label{lem:E8-bound}
  Let $x \in Y_1(E_8)(\N)$
  with $\calo_{\Z/16\Z}(x) \subset Y_1(E_8)(\Z_{\geq 2})$.
  Then there exists an $x' \in \calo_{\Z/16\Z}(x)$
  such that $\pi_4(x') < 16966221628$.
\end{lemma}

\begin{proof}
    Let $F = (F_{i, j})_{i, j} = (\pi_i(\sigma_{j - 1}(x)))_{i, j}$
    be the frieze corresponding to $x$.
    Since $\calo_{\Z/16\Z}(x) \subset Y_1(E_8)(\Z_{\geq 2})$,
    we have that $\pi_i(\sigma_j(x)) \in \Z_{\geq 2}$ for all $i, j$,
    i.e. all entries of the frieze are at least $2$.
    Define $v := (\frac{1}{16} \log_2(\prod_{j = 1}^{16} F_{i, j}))_i$
    to be the vector of averages of $2$-logarithms
    of the rows of the frieze $F$.
    As in \cite[Example 3.1]{muller},
    the $i$-th coordinate $v_i$ is bounded above by
    \[
        v_i \leq \sum_{j=1}^8 c^{-1}_{j, i} \log_2\paren{2^{\sum_{k \neq j} c_{k,j}} + 1}.
    \]
    
    Observe that if all the entries of $F$ are at least $2$,
    then $F_{3, k}, F_{4, k}, F_{7, k} \geq 3$
    for all $k$ due to the first, second, and eighth $E_8$-frieze equations.
    The $E_8$ mesh relations are:
    \begin{align*}
        F_{1, k} F_{1, k + 1} &= 1 + F_{4, k}
            &F_{5, k} F_{5, k + 1} &= 1 + F_{4, k + 1} F_{6, k} \\
        F_{2, k} F_{2, k + 1} &= 1 + F_{3, k}
            &F_{6, k} F_{6, k + 1} &= 1 + F_{5, k + 1} F_{7, k} \\
        F_{3, k} F_{3, k + 1} &= 1 + F_{2, k + 1} F_{4, k}
            &F_{7, k} F_{7, k + 1} &= 1 + F_{6, k + 1} F_{8, k} \\
        F_{4, k} F_{4, k + 1} &= 1 + F_{1, k + 1} F_{3, k + 1} F_{5, k}
            &F_{8, k} F_{8, k + 1} &= 1 + F_{7, k + 1}.
    \end{align*}
    Adapting the steps of \cite[Section 2]{muller}, we obtain even better
    bounds on each $v_i$ except $v_7$ (since it does not involve
    the third, fourth, nor seventh rows).
    Consider the square of the product of all entries in each row
    and apply the $E_8$ mesh relations:
    \begin{align*}
        \prod_{k = 1}^{16} F_{1, k}^2
            &= \prod_{k = 1}^{16} (1 + F_{4, k})
            \leq \prod_{k = 1}^{16} (3^{-1} + 1) F_{4, k} \\
        \prod_{k = 1}^{16} F_{2, k}^2
            &= \prod_{k = 1}^{16} (1 + F_{3, k})
            \leq \prod_{k = 1}^{16} (3^{-1} + 1) F_{3, k} \\
        \prod_{k = 1}^{16} F_{3, k}^2
            &= \prod_{k = 1}^{16} (1 + F_{2, k + 1} F_{4, k})
            \leq \prod_{k = 1}^{16} (6^{-1} + 1) F_{2, k + 1} F_{4, k} \\
        \prod_{k = 1}^{16} F_{4, k}^2
            &= \prod_{k = 1}^{16} (1 + F_{1, k + 1} F_{3, k + 1} F_{5, k})
            \leq \prod_{k = 1}^{16} (12^{-1} + 1) F_{1, k + 1} F_{3, k + 1} F_{5, k} \\
        \prod_{k = 1}^{16} F_{5, k}^2
            &= \prod_{k = 1}^{16} (1 + F_{4, k + 1} F_{6, k})
            \leq \prod_{k = 1}^{16} (6^{-1} + 1) F_{4, k + 1} F_{6, k} \\
        \prod_{k = 1}^{16} F_{6, k}^2
            &= \prod_{k = 1}^{16} (1 + F_{5, k + 1} F_{7, k})
            \leq \prod_{k = 1}^{16} (6^{-1} + 1) F_{5, k + 1} F_{7, k} \\
        \prod_{k = 1}^{16} F_{8, k}^2
            &= \prod_{k = 1}^{16} (1 + F_{7, k + 1})
            \leq \prod_{k = 1}^{16} (3^{-1} + 1) F_{7, k + 1}.
    \end{align*}

    With
    \[
        B_j :=
            \begin{cases}
                4 & \text{ if } j = 7, \\
                3 & \text{ if } j = 1, 2, 8, \\
                6 & \text{ if } j = 3, 5, 6, \\
                12 & \text{ if } j = 4,
            \end{cases}
    \]
    one obtains the bounds
    \[
        0 < (C_{E_8} \cdot v)_j \leq \log_2\paren{B_j^{-1} + 1}.
    \]
    As noted in \cite[Proposition 2.3]{muller},
    the positivity of the entries in $C_{E_8}^{-1}$
    gives an upper bound for $v_i$ in terms of the bounds on
    $(C_{E_8} \cdot v)_i$:
    \[
        v_i = \paren{C_{E_8}^{-1} \paren{C_{E_8} v}}_i
            \leq \sum_{j = 1}^8 c^{-1}_{j, i} \log_2\paren{B_j^{-1} + 1}.
    \]
    In particular, $2^{v_4} < 16966221628$ and
    \[
        \prod_{j=1}^{16} F_{4, j} 
            \leq \prod_{j = 1}^8 \paren{B_j^{-1} + 1}^{16 c^{-1}_{j, 4}}
            < 16966221628^{16}.
    \]
    Take the frieze $F'$ in the $(\Z/16\Z)$-orbit of $F$
    such that $F'_{4, 1} \leq F'_{4, j}$ for all
    $1 \leq j \leq 16$,
    and let $x'$ denote the point in $Y_1(E_8)(\Z_{\geq 2})$ corresponding
    to $F'$.
    Since $\pi_4(x') = F'_{4, 1}$ is the smallest factor in the
    product of $16$ terms,
    we have that $\pi_4(x') < 16966221628$.
\end{proof}

\begin{proof}[Proof of Theorem \ref{thm:positive-siegel}(2) in the $E_8$ case]
    By Lemma \ref{lem:friezes-at-least-two},
    it is sufficient to rule out $E_7$-friezes
    whose entries are all at least $2$
    and to show that there are exactly $4$ points
    $x \in Y_1(E_8)(\N)$ whose entire $(\Z/16\Z)$-orbit $\calo_{\Z/16\Z}(x)$
    lies in $Y_1(E_8)(\Z_{\geq 2})$.
    By Lemma \ref{lem:E8-bound},
    for such an $x \in Y_1(E_8)$, there exists an element
    $x' \in \calo_{\Z/16\Z}(x) \subset Y_1(E_8)(\Z_{\geq 2})$
    such that $\pi_4(x') < 16966221628$.
    But Lemma \ref{lem:E8-x4-largest} says that the other coordinates
    of $x'$ are at most $\pi_4(x') + 4 < 16966221632$.
    Hence for every $x \in Y_1(E_8)$ such that $\calo_{\Z/16\Z}(x)
    \subset Y_1(E_8)(\Z_{\geq 2})$, there exists an element
    \[
      x' \in \calo_{\Z/16\Z}(x) \cap Y_1(E_8)(\N) \cap \set{x \in \A_\C^{16} \Mid x_i, y_i < 16966221632 \text{ for all } i}.
    \]

    A computer search\footnote{This took a few dozen hours on
    optimized C++ programs on $64$ NVIDIA V100 GPUs
    (via CUDA) and $1000$ x86 CPU cores.}
    exhaustively finds $26952$ points in
    \[
      Y_1(E_8)(\N) \cap
        \set{x \in \A_\C^{16} \Mid x_i, y_i < 16966221632 \text{ for all } i},
    \]
    the largest of which is
    \[
        (1320, 165, 16994, 2820839, 134632, 6433, 461, 21;
            2137, 103, 166, 8, 21, 21, 14, 22).
    \]
    Searching among these $26952$ points finds exactly $4$ points
    whose entire orbit is in $Y_1(E_8)(\Z_{\geq 2})$:
    \begin{itemize}
        \item $(6, 4, 11, 29, 21, 13, 5, 2; 5, 3, 3, 3, 2, 2, 3, 3)$,
        \item $(5, 3, 8, 29, 18, 7, 3, 2; 6, 3, 4, 2, 2, 3, 3, 2)$,
        \item $(6, 3, 11, 41, 16, 7, 5, 3; 7, 4, 4, 2, 3, 3, 2, 2)$,
        \item $(7, 4, 15, 41, 18, 13, 8, 3; 6, 4, 3, 3, 3, 2, 2, 3)$.
    \end{itemize}

    It remains to verify the $E_7$ hypothesis of
    Lemma \ref{lem:friezes-at-least-two}.
    Suppose that an $E_7$-frieze has all entries at least $2$.
    By Lemma \ref{lem:xn-bound}, it has a horizontal translate
    whose corresponding point satisfies $x_4 \leq c_{4, E_7} < 135503$.
    Lift this point to $Y_1(E_8)(\N) \cap \{x_8 = 1\}$
    using Proposition \ref{prop:hyperplane}.
    The lift satisfies the hypotheses of Lemma \ref{lem:E8-x4-largest},
    so all of its coordinates are less than $135507$
    and it therefore occurs in the finite search above.
    Exactly $4400$ of the searched points in $Y_1(E_8)(\N)$
    satisfy $x_8 = 1$
    and none of the points has its entire $(\Z/10\Z)$-orbits
    contained in $Y_1(E_7)(\Z_{\geq 2})$.
    Hence there are no $E_7$-friezes with all entries at least $2$,
    so Lemma \ref{lem:friezes-at-least-two}(b) now gives the $E_8$ count.
\end{proof}

\begin{remark}
  The four points of $Y_1(E_8)(\N)$ whose $(\Z/16\Z)$-orbits
  lie in $Y_1(E_8)(\Z_{\geq 2})$ actually lie in the same
  orbit of order $4$. They correspond to
  the four $E_8$-friezes with all entries at least $2$
  given by horizontal translations of the frieze in
  Example \ref{ex:E8}.
  Note that this affirmatively resolves \cite[Conjecture 5.14]{bfgst}.
\end{remark}

\subsection{Reduction from \texorpdfstring{$E_8$}{E8} to smaller types}
Using the reduction principle of
Proposition \ref{prop:hyperplane},
the affine variety $Y_1(\Delta_n)$ for these Dynkin types
can be obtained as the intersection of $Y_1(E_8)$
with hyperplanes as follows:
\begin{align*}
    Y_1(A_n) &\cong Y_1(E_8) \cap
        \paren{\bigcap_{i=1}^{8-n} \{x_i = 1\}}
        \qquad \text{for $n \leq 7$}, \\
    Y_1(D_n) &\cong Y_1(E_8) \cap \{x_2 = 1\} \cap
        \paren{\bigcap_{i = 0}^{6 - n} \{x_{8 - i} = 1\}}
        \qquad \text{for $4 \leq n \leq 7$}, \\
    Y_1(E_6) &\cong Y_1(E_8) \cap \{x_7 = 1\} \cap \{x_8 = 1\} \\
    Y_1(E_7) &\cong Y_1(E_8) \cap \{x_8 = 1\}.
\end{align*}
Since we have computationally found all $26952$ points in $Y_1(E_8)(\N)$,
we can directly verify the other $ADE$-type formulae for $n \leq 7$
by taking the corresponding subsets of the $26952$ elements of $Y_1(E_8)(\N)$.
The other finite types of rank $n \leq 7$ follow from the $ADE$ types by
the observations of \cite[Section 4]{fontaine-plamondon}
on foldings of Dynkin diagrams.

\begin{proof}[Proof of Theorem \ref{thm:positive-siegel}(2) in the $E_7$ case]
  In particular, we directly verify that
  there are exactly $4400$ points in $Y_1(E_8)(\N)$
  whose eighth coordinate is $1$.
\end{proof}
\begin{remark}
  We can also directly verify that none of the points
  in $Y_1(E_7)(\N)$ have $(\Z/10\Z)$-orbits in $Y_1(E_7)(\Z_{\geq 2})$,
  affirmatively resolving \cite[Conjecture 5.11]{bfgst}.
\end{remark}


\section{Effective height bounds on positive integral points in finite type}
For a point $x = (x_1, \ldots, x_n; y_1, \ldots, y_n)
\in Y_1(\Delta_n)(\N)$, define its height by
\[
  H(x) := \max\{x_1, \ldots, x_n, y_1, \ldots, y_n\}.
\]
In this section, we establish general effective bounds on
$Y_1(\Delta_n)(\N)$ in terms of the following $\Z$-lattice in a hypercube:
\[
  S(N_1, N_2) := \set{(x_1, \ldots, x_n; y_1, \ldots, y_n) \in \Z^{2n} \Mid
    N_1 \leq x_i, y_i \leq N_2 \text{ for all } i}.
\]
This gives a further refinement of Theorem \ref{thm:positive-siegel}(2)
by bounding the height of positive integral points;
these bounds are sharp for $A_n$, $C_n$,
and the exceptional Dynkin types
(we will discuss sharp height bounds for $B_n$ and $D_n$
in \cite{zhang-maximal-dynkin}).
\begin{proposition}
  \label{prop:explicit-bounds}
  Let $\Delta_n$ be a finite Dynkin type.
  For the following Dynkin types,
  the positive integral points $Y_1(\Delta_n)(\N)$
  are contained in $S(1, N_{\Delta_n})$, where
  \begin{multicols}{3}
    \begin{itemize}
      \item $N_{A_n} = F_{n+2}$;
      \item $N_{C_n} = F_{2n+1}$; \\
      \item $N_{E_6} = 307$;
      \item $N_{E_7} = 6425$;
      \item $N_{E_8} = 2820839$;
      \item $N_{F_4} = 307$;
      \item $N_{G_2} = 14$;
    \end{itemize}
  \end{multicols}
  \noindent
  and $F_n$ denotes the $n$-th Fibonacci number.

  For the remaining two Dynkin types,
  the positive integral points $Y_1(\Delta_n)(\N)$
  are contained in $S(1, N_{\Delta_n}^{P_{\Delta_n}})$, where
  \begin{multicols}{2}
    \begin{itemize}
      \item $N_{B_n} = 2^{\frac{n^2+n-2}{2}}$, $P_{B_n} = n + 1$;
      \item $N_{D_n} = 2^{\frac{n^2}{2}}$,
        $P_{D_n} = 2n$.
    \end{itemize}
  \end{multicols}
\end{proposition}

As before, let $C_{\Delta_n} = (c_{i, j})$ be a Cartan matrix
for Dynkin type $\Delta_n$. Denote its inverse by
$C_{\Delta_n}^{-1} = (c^{-1}_{i,j})$.
Define the following associated values:
\begin{align*}
    b_{i, \Delta_n} &:= \prod_{j=1}^{n} 2^{c^{-1}_{j, i}} \\
    c_{i, \Delta_n} &:=
        \prod_{j=1}^{n}
        \paren{1 + 2^{\sum_{k \neq j} c_{k,j}}}^{c^{-1}_{j, i}}.
\end{align*}
Muller \cite[Proposition 2.3 and Example 3.1]{muller}
gives bounds on frieze entries in terms of
these values.
\begin{lemma}[{\cite[Proposition 2.3 and Example 3.1]{muller}}]
    \label{lem:xn-bound-muller}
    Let $\Delta_n$ be a finite Dynkin type of rank $n$
    and let $F$ be a positive integral $\Delta_n$-frieze
    of period $P_{\Delta_n}$.
    Then there is the following upper bound on
    the product of entries in its $i$-th row:
    $\prod_{j = 1}^{P_{\Delta_n}} F_{i, j} \leq b_{i, \Delta_n}^{P_{\Delta_n}}$.
    Furthermore, if all entries in $F$ are at least $2$,
    then $\prod_{j = 1}^{P_{\Delta_n}} F_{i, j} \leq c_{i, \Delta_n}^{P_{\Delta_n}}$.
\end{lemma}
\begin{remark}
    While friezes can be defined for
    more general Dynkin types, the proof of the following
    result crucially uses the result of Lusztig--Tits
    \cite{lusztig-tits} that the inverse of a Cartan matrix
    $C_{\Delta_n}$ is positive if and only if
    $\Delta_n$ is finite type.
    We also note a misprint in \cite[Example 3.1]{muller}: the bound
    $(\frac{151875}{16384})^{16} \approx 2^{51}$ on the eighth row
    of an $E_8$-frieze should be
    $c_{8, E_8}^{16} =
    (\frac{177347025604248046875}{144115188075855872})^{16}
    \approx 2^{164}$.
\end{remark}
An immediate consequence of Lemma \ref{lem:xn-bound-muller}
and Proposition \ref{prop:correspondence} is that if the
frieze $F$ corresponds to the point
$(x_1, \ldots, x_n; y_1, \ldots, y_n) \in Y_1(\Delta_n)(\N)$,
then $x_i \leq  b_{i, \Delta_n}^{P_{\Delta_n}}$.
Furthermore, if all entries in $F$ are at least $2$, then
$x_i \leq c_{i, \Delta_n}^{P_{\Delta_n}}$.
All of the $y_i$ appear as frieze entries
(see \cite[Section 5.3]{dhl}),
so a bound valid for every frieze entry also bounds all of the
$x_i$ and $y_i$.

The Diophantine model of friezes
allows us to find a minimal element in each $(\Z/P_{\Delta_n}\Z)$-orbit
on which we can reduce existing bounds by a power of
$P_{\Delta_n}^{-1}$.
\begin{lemma}
    \label{lem:xn-bound}
    Let $\Delta_n$ be a finite Dynkin type,
    let $F$ be a positive integral $\Delta_n$-frieze,
    and fix an $i \in \{1, \ldots, n\}$.
    There is a $\Delta_n$-frieze $F'$ corresponding to
    $(x_1, \ldots, x_n; y_1, \ldots, y_n) \in Y_1(\Delta_n)(\N)$
    such that $F$ is a horizontal translation of $F'$
    and $x_i \leq b_{i, \Delta_n}$.
    Furthermore, if all entries in $F$ are at least $2$,
    then $x_i \leq c_{i, \Delta_n}$.
\end{lemma}

\begin{proof}
    Each $\Delta_n$-frieze $F$ is a horizontal translation of
    a frieze $F'$ in its $(\Z/P_{\Delta_n}\Z)$-orbit
    such that $F'_{i, 1} \leq F'_{i, j}$ for all
    $1 \leq j \leq P_{\Delta_n}$.
    By Lemma \ref{lem:xn-bound-muller},
    there is a bound on
    $\prod_{j = 1}^{P_{\Delta_n}} F'_{i, j}
    \leq b_{i, \Delta_n}^{P_{\Delta_n}}$.
    But $x_i = F'_{i, 1}$ is the smallest factor in the
    product of $P_{\Delta_n}$-many terms,
    hence $x_i \leq b_{i, \Delta_n}$.
    If the entries of $F$ are all at least $2$,
    then the entries of the translation $F'$ are also all at least $2$.
    Hence the same argument gives the bound in terms of $c_{i, \Delta_n}$
    instead of $b_{i, \Delta_n}$.
\end{proof}

We conclude this section by proving Proposition \ref{prop:explicit-bounds}
by determining the largest column sum
and computing the bounds $b_{i, \Delta_n}$ and $c_{i, \Delta_n}$
for each finite Dynkin type.
\begin{proof}[Proof of Proposition \ref{prop:explicit-bounds}]
  We explicitly compute the largest $\log_2(b_i) = \sum_{j=1}^n {c^{-1}_{j, i}}$
  for each Dynkin type. The Cartan matrix that we use for each type
  is listed in \nameref{sec:table-A}.

  For $A_n$, the inverse of the Cartan matrix is given by:
  $c^{-1}_{i, j} = \min(i, j) - \frac{ij}{n + 1}$.
  When $n$ is even, simple manipulations show that
  $\log_2(b_i)$ attains its maximum at $i=\frac{n}{2}$
  and $\log_2(b_{\frac{n}{2}}) = \frac{n(n+2)}{8}$.
  When $n$ is odd, $\log_2(b_i)$ attains its maximum at $i=\frac{n+1}{2}$
  and $\log_2(b_{\frac{n+1}{2}}) = \frac{(n+1)^2}{8}$.
  This gives a bound $\max_i b_{i,A_n} \leq 2^{\frac{(n+1)^2}{8}}$,
  but in fact a sharp bound of the $(n+2)$-nd Fibonacci number
  $F_{n+2}$ was already achieved by \cite[Theorem 1(1)]{cheah-dsg}.

  For $B_n$, the inverse of the Cartan matrix is given by:
  \[
    c^{-1}_{i, j} =
      \begin{cases}
        \frac{n-i+1}{2} \text{ if } j = 1, \\
        n + 1 - \max(i, j) \text{ if } j > 1.
      \end{cases}
  \]
  The entries of $C_{B_n}^{-1}$ are clearly the largest in column $2$,
  so $\log_2(b_2) = n - 1 + \sum_{j=2}^n (n - j + 1) = \frac{n^2 + n - 2}{2}$.

  For $C_n$, the Cartan matrix is the transpose of $C_{B_n}$ so
  its inverse is given by:
  \[
    c^{-1}_{i, j} =
      \begin{cases}
        \frac{n-j+1}{2} \text{ if } i = 1, \\
        n + 1 - \max(i, j) \text{ if } i > 1.
      \end{cases}
  \]
  The entries of $C_{C_n}^{-1}$ are clearly the largest in column $1$,
  so $\log_2(b_1) = \frac{n}{2} + \sum_{j=2}^n (n - j + 1) = \frac{n^2}{2}$.
  This gives $\max_i b_{i,C_n} = 2^{\frac{n^2}{2}}$,
  but in fact a sharp bound of the $(2n+1)$-st Fibonacci number
  $F_{2n+1}$ was already achieved
  by \cite[Theorem 1(2)]{cheah-dsg}.

  For $D_n$, the inverse of the Cartan matrix is given by:
  \[
    c^{-1}_{i, j} =
      \begin{cases}
        \frac{n}{4} \text{ if } i = j \in \{1, 2\}, \\
        \frac{n-2}{4} \text{ if } i = 1, j = 2 \text{ or } i = 2, j = 1 \\
        \frac{n - j + 1}{2} \text{ if } i \leq 2, j \geq 3 \\
        \frac{n - i + 1}{2} \text{ if } i \geq 3, j \leq 2 \\
        n + 1 - \max(i, j) \text{ if } i \geq 3, j \geq 3.
      \end{cases}
  \]
  We can directly calculate that $\log_2(b_1) = \log_2(b_2) = \frac{n(n-1)}{4}$
  while $\log_2(b_i) = \frac{n^2 - n - i^2 + 3i - 2}{2}$ for $i \geq 3$.
  The $i \geq 3$ expression is decreasing in $i$, so
  we only need to compare $\log_2(b_1) = \log_2(b_2)$ with
  $\log_2(b_3) = \frac{n^2 - n - 2}{2} = \frac{(n+1)(n-2)}{2}$.
  Clearly, $\frac{(n+1)(n-2)}{2} > \frac{n(n-1)}{4}$ when $n \geq 3$.

  For the exceptional types $E_6$, $E_7$, $E_8$, $F_4$, and $G_2$,
  we can use a computer algebra system to compute the sums
  of the columns of the inverses of each Cartan matrix:
  \begin{itemize}
    \item $\lfloor b_{4, E_6} \rfloor = 2097152$,
    \item $\lfloor b_{4, E_7} \rfloor = 281474976710656$,
    \item $\lfloor b_{4, E_8} \rfloor = 43556142965880123323311949751266331066368$,
    \item $\lfloor b_{3, F_4} \rfloor = 2097152$,
    \item $\lfloor b_{2, G_2} \rfloor = 32$.
  \end{itemize}
  For the exceptional types of rank $\leq 6$, a short search
  finds all of the points in $Y_1(\Delta_n)(\N)$
  and the actual maxima given in the theorem statement.
  The explicit searches in Section \ref{sec:En}
  give the maxima for $E_7$ and $E_8$.
\end{proof}


\section{Positive Mordell--Schinzel in infinite type}
\label{sec:positive-mordell-schinzel}

It is a simple exercise
to show that there are exactly $5$ positive integer
solutions to the Diophantine equation
\begin{align*}
  S_1: \, x y z &= x + y + 1.
\end{align*}
Similarly, there are exactly $6$ positive integer solutions to
the Diophantine equation
\begin{align*}
  S_2: \, x y z &= x^2 + y + 1.
\end{align*}
Mohanty \cite[Theorem 2]{mohanty-1977}
proved that there are exactly $9$ positive integer solutions to
the Diophantine equation
\begin{align*}
  S_3: \, x y z &= x^3 + y + 1;
\end{align*}
we give another proof of this fact in Proposition \ref{prop:G2}.
The purpose of this section is to show how the theory of
friezes and cluster algebras of \textit{infinite} type
can establish the infinitude of
positive integral solutions to Diophantine equations,
using the example of $S_m: x y z = x^m + y + 1$
with degree $m \geq 4$.

In this section,
we construct injections on positive integral points
from frieze modular varieties $Y_1(C)$
and use Theorem \ref{thm:positive-siegel}
to prove Corollary \ref{cor:rank-2-3}.
Rank-$2$ types
are classified by generalized Cartan matrices of the form
$\begin{psmallmatrix}
 2 & -b \\
 -a & 2
\end{psmallmatrix}$
for positive integers $a$ and $b$;
they are of infinite type precisely when $ab \geq 4$.
The surface $S_m$
corresponds to the special case $(a, b) = (m, 1)$,
and in this case the map is an isomorphism,
for every $m \geq 1$.
The threefold analogue arises from rank-$3$ infinite types
with generalized Cartan matrices of the form
$\begin{psmallmatrix}
 2 & -c & 0 \\
 -a & 2 & -b \\
 0 & -d & 2
\end{psmallmatrix}$
for positive integers $a$, $b$, $c$, and $d$
such that $abcd \geq 3$.
Part (2) of Corollary \ref{cor:rank-2-3}
is the special case $c = d = 1$.

Before moving on to the proof of Corollary \ref{cor:rank-2-3},
we remark that it can be viewed in a \textit{positive} refinement
to the program of Mordell--Schinzel
(as well as Siegel, Corvaja--Zannier, and Vojta
\cite{siegel,siegel-translated,corvaja-zannier-2004, corvaja-zannier-2010};
see Section \ref{sec:rational-functions}
and Remark \ref{rem:G2}).
Jacobsthal \cite{jacobsthal}, Barnes \cite{barnes}, Mills \cite{mills-1954},
and Schinzel \cite{schinzel-2015,schinzel-2018}
found counterexamples to Mordell's claim that
Equation \ref{eq:mordell-schinzel} has infinitely many integer solutions
when $3 \leq \deg_x(G) \cdot \deg_y(G) \leq 4$,
but several cases have been recovered
by Schinzel \cite{schinzel-2015,schinzel-2018}
and Koll\'{a}r--Li \cite[Theorem 3]{kollar-li}.
In particular, Schinzel \cite[Theorem 3]{schinzel-2015}
established that if $m \geq 4$ and $a, b, c$ are nonzero integers,
then there are infinitely many integer solutions $(x, y, z)$ to
the Diophantine equation
\[
  xyz = ax^m + by + c
\]
with $\gcd(y, c) = 1$.
Corollary \ref{cor:rank-2-3}(1)
gives a new proof of the $a=b=c=1$ case.
The Mordell--Schinzel conjecture,
as formulated by Koll\'{a}r--Li \cite[Conjecture 2]{kollar-li},
predicts that there are infinitely many integral solutions to
the Equation \ref{eq:mordell-schinzel} when
$\deg_x(G) \geq 3$ and $\deg_y(G) \geq 3$.
It is unclear if these Diophantine equations
should also have infinitely many \textit{positive} integral solutions,
but perhaps cluster algebras of infinite type
should inform the formulation of a \textit{positive} Mordell--Schinzel conjecture.

\subsection{The general rank \texorpdfstring{$2$}{2} case}
We prove Corollary \ref{cor:rank-2-3}(1)
as the special case $b = 1$ of the following construction.
Consider the generalized $2 \times 2$ Cartan matrix
\[
  C = \begin{pmatrix}
         2 & -b \\
         -a & 2
       \end{pmatrix}.
\]
$C$ is symmetrizable for all positive integers $a$ and $b$.
The $C$-frieze equations are:
\begin{align*}
  x_1 y_1 &= x_2^a + 1, \\
  x_2 y_2 &= x_1^b + 1.
\end{align*}

Applying the substitution $x_1 = \frac{x_2^a + 1}{y_1}$
to the second equation and relabeling
\begin{align*}
  x_2 \longmapsto x, \qquad
  y_1 \longmapsto y, \qquad
  y_2 \longmapsto z.
\end{align*}
yields the equation:
\begin{align}
  \label{eq:2-surface}
  x y^b z &= (x^a + 1)^b + y^b.
\end{align}

If $(a, b) \notin \{(1, 1), (1, 2), (2, 1), (1, 3), (3, 1)\}$,
then $C$ is of infinite type and $t_C \leq 0$.
The assumption that $ab \geq 4$ guarantees that $C$ is of infinite type,
so there are infinitely many positive integral solutions
to the $C$-frieze equations by Theorem \ref{thm:positive-siegel}(1).
The map
\[
  (x_1, x_2; y_1, y_2) \longmapsto (x_2, y_1^b, y_2)
\]
is injective on positive integral points:
$y_1$ is recovered as the positive $b$-th root of the second coordinate,
and then $x_1 = (x_2^a + 1)/y_1$.
Thus there are infinitely many positive integral solutions
of the equation
\begin{align*}
  x y z &= (x^a + 1)^b + y.
\end{align*}
When $b = 1$, the inverse map is
\[
  (x, y, z) \longmapsto (xz - 1, x; y, z).
\]
It preserves positive integral points, since
$y(xz - 1) = x^a + 1 > 0$.
For $a = 1$, $2$, or $3$, the finite counts
follow from Propositions \ref{prop:A2}, \ref{prop:B2}, and \ref{prop:G2},
respectively.
This proves Corollary \ref{cor:rank-2-3}(1).

\begin{remark}
  Our infinitude argument for Corollary \ref{cor:rank-2-3}(1)
  is written such that it may be generalized
  whenever one finds cluster structure underlying
  Diophantine equations,
  but it can also be written in elementary terms
  in this specific instance.
  Write $F(t) = (t^a + 1)^b$ and suppose that
  $y(xz - 1) = F(x)$ is a positive integral solution with $x < z$.
  Set $d = xz - 1$.
  Since $xz \equiv 1 \pmod d$,
  \[
    F(z) \equiv z^{ab} F(x) \equiv 0 \pmod d.
  \]
  Moreover, $d \equiv -1 \pmod z$ and $F(z) \equiv 1 \pmod z$,
  so
  \[
    z' := \frac{F(z)/d + 1}{z}
  \]
  is a positive integer, and $(z, d, z')$ is another solution.
  If $ab \geq 4$, then $d < z^2$ and $F(z) > z^{ab}$,
  whence $F(z)/d > z^2$ and $z' > z$.
  Starting with $(1, 1, 2^b + 1)$ therefore gives
  infinitely many positive integral solutions.
  On the image of $Y_1(C)(\N)$, we have $d = x_1^b$.
  Mutating first at $x_2$ and then at $x_1$
  gives the cluster $(u, y_2)$, where
  $u = (y_2^a + 1)/x_1$.
  Its image under the same construction is
  \[
    \left(y_2, x_1^b, \frac{u^b + 1}{y_2}\right),
  \]
  which is exactly $(z, d, z')$.
  Thus the elementary recurrence is the rank-$2$ expression
  of the cluster mutations used in our proof.
  The construction also gives the rank-$3$ application below
  and suggests analogous constructions from other cluster algebras.
\end{remark}


\subsection{A rank \texorpdfstring{$3$}{3} case}
We prove Corollary \ref{cor:rank-2-3}(2)
as the special case $c = d = 1$ of the following construction.
Consider the generalized $3 \times 3$ Cartan matrix
\[
  C = \begin{pmatrix}
         2 & -c & 0 \\
         -a & 2 & -b \\
         0 & -d & 2
       \end{pmatrix}.
\]
$C$ is symmetrizable for all positive integers $a$, $b$, $c$, and $d$.
The $C$-frieze equations are:
\begin{align*}
  x_1 y_1 &= x_2^a + 1, \\
  x_2 y_2 &= x_1^c + x_3^d, \\
  x_3 y_3 &= x_2^b + 1.
\end{align*}
Applying the substitutions
$x_1 = \frac{x_2^a + 1}{y_1}$
and $x_3 = \frac{x_2^b + 1}{y_3}$
to the second equation yields:
\begin{align*}
  x_2 y_2 y_1^c y_3^d &= (x_2^a + 1)^c y_3^d + (x_2^b + 1)^d y_1^c.
\end{align*}
Relabeling
\begin{align*}
  x_2 \longmapsto x, \qquad
  y_1 \longmapsto z, \qquad
  y_2 \longmapsto w, \qquad
  y_3 \longmapsto y.
\end{align*}
yields the equation:
\begin{align*}
  x y^d z^c  w &= (x^a + 1)^c y^d + (x^b + 1)^d z^c.
\end{align*}

Notice that if $abcd \geq 3$, then $C$ is of infinite type
and $t_C \leq 0$.
There are infinitely many positive integral solutions
to the $C$-frieze equations by Theorem \ref{thm:positive-siegel}(1).
The map
\[
  (x_1, x_2, x_3; y_1, y_2, y_3)
    \longmapsto (x_2, y_3^d, y_1^c, y_2)
\]
is injective on positive integral points:
$y_1$ and $y_3$ are recovered as positive roots,
and then $x_1$ and $x_3$ are recovered from the outer frieze equations.
Thus there are infinitely many positive integral solutions
of the equation
\begin{align*}
  x y z w &= (x^a + 1)^c y + (x^b + 1)^d z.
\end{align*}
Taking $c = d = 1$
gives Corollary \ref{cor:rank-2-3}(2).

\appendix

\section{Elementary proofs for small
\texorpdfstring{$n$}{n}}
\label{sec:appendix}

Enumeration theorems for all Dynkin types for small enough $n$
can be proved in the same manner as Theorem \ref{thm:positive-siegel}(2)
using the bounds provided by Proposition \ref{prop:explicit-bounds}.
In this appendix, we instead illustrate that elementary arithmetic methods,
relying solely on simple divisibility relations
from the Diophantine equations defining $Y_1(\Delta_n)$,
can directly prove the enumeration theorems
for various $\Delta_n$ of small $n$.

\subsection{Frieze modular varieties of type \texorpdfstring{$A_n$}{An} for \texorpdfstring{$n < 4$}{n < 4}}

The classical frieze enumeration theorem of
Conway--Coxeter \cite{coxeter-conway-2}
shows that the number of positive integral friezes of type $A_n$
is precisely the $(n + 1)$-st Catalan number
by constructing a bijection between the $A_n$ friezes
and the triangulations of an $(n + 3)$-gon.
By Proposition \ref{prop:correspondence},
the number of positive integral points on $Y_1(A_n)$
is precisely the $(n + 1)$-st Catalan number. 

We give another proof for small $n$
by directly counting the points of $Y_1(A_n)(\N)$.

\begin{proposition}
  The affine curve $Y_1(A_1)$ has
  exactly $2$ positive integral points.
\end{proposition}

\begin{proof}
    The $A_1$-frieze equation is:
    \begin{align*}
      x_1 y_1 &= 2.
    \end{align*}
    Since $x_1$ and $y_1$ are positive integers,
    the only possibilities are $(1, 2)$ and $(2, 1)$.
\end{proof}

\begin{proposition}
  \label{prop:A2}
  The affine surface $Y_1(A_2)$ has
  exactly $5$ positive integral points:
  \begin{multicols}{3}
    \begin{itemize}
      \small
      \item $(1, 1, 2, 2)$,
      \item $(1, 2, 3, 1)$,
      \item $(2, 1, 1, 3)$,
      \item $(2, 3, 2, 1)$,
      \item $(3, 2, 1, 2)$.
    \end{itemize}
  \end{multicols}
\end{proposition}

\begin{proof}
    The $A_2$-frieze equations are:
    \begin{align*}
      x_1 y_1 &= x_2 + 1, \\
      x_2 y_2 &= x_1 + 1.
    \end{align*}
    Observe that $Y_1(A_2)(\N)$
    is in bijection
    with the set of positive integral points on
    the smooth affine cubic surface,
    \begin{align*}
    x y z &= x + y + 1,
    \end{align*}
    under the change of variables
    \begin{align*}
    x_1 \longmapsto x, \qquad
    x_2 \longmapsto \frac{x_1 + 1}{y_2}, \qquad
    y_1 \longmapsto z, \qquad
    y_2 \longmapsto y.
    \end{align*}
    
    Observe that $x_2$ is congruent to $-1 \pmod {x_1}$
    from the initial $A_2$-frieze equations
    and that $y = y_2$ is congruent to $-1 \pmod{x_1}$
    from the affine cubic equations.
    Writing $x_2 = mx_1 - 1$ and $y_2 = nx_1 - 1$,
    we have that $(mx_1 - 1) (nx_1 - 1) = x_1 + 1$
    so $mnx_1^2 - (m + n + 1)x_1 = 0$.
    Since $x_1 \neq 0$, we have
    $x_1 = \frac{m + n + 1}{mn}$.
    This is a positive integer only 
    if $m, n \in \set{1, 2, 3}$.
    These correspond to the $5$ positive integral solutions
    $(x_1, x_2; y_1, y_2) \in Y_1(A_2)(\N)$
    listed.
\end{proof}

\begin{remark}
    The affine cubic $x y z = x + y + 1$
    has infinite families of integer points
    $(t, -1, -1)$, $(-1, t, -1)$, and $(t, -t - 1, 0)$.
    So $\Delta_n = A_2$ is the first Dynkin type
    with $\#Y_1(\Delta_n)(\Z) = \infty$ but
    $\#Y_1(\Delta_n)(\N) < \infty$.
\end{remark}

\begin{proposition}
  The affine three-fold $Y_1(A_3) \cong Y_1(D_3)$ has
  exactly $14$ positive integral points:
  \begin{multicols}{3}
    \begin{itemize}
      \small
      \item $(1, 1, 1, 2, 2, 2)$
      \item $(1, 1, 2, 2, 3, 1)$,
      \item $(1, 2, 1, 3, 1, 3)$,
      \item $(1, 2, 3, 3, 2, 1)$,
      \item $(1, 3, 2, 4, 1, 2)$,
      \item $(2, 1, 1, 1, 3, 2)$,
      \item $(2, 1, 2, 1, 4, 1)$,
      \item $(2, 3, 1, 2, 1, 4)$,
      \item $(2, 3, 4, 2, 2, 1)$,
      \item $(2, 5, 3, 3, 1, 2)$,
      \item $(3, 2, 1, 1, 2, 3)$,
      \item $(3, 2, 3, 1, 3, 1)$,
      \item $(3, 5, 2, 2, 1, 3)$,
      \item $(4, 3, 2, 1, 2, 2)$.
    \end{itemize}
  \end{multicols}
\end{proposition}

\begin{proof}
    The $A_3$-frieze equations
    (equivalently the $D_3$-frieze equations) are: 
    \begin{align*}
      x_1 y_1 &= x_2 + 1, \\
      x_2 y_2 &= x_1 + x_3, \\
      x_3 y_3 &= x_2 + 1.
    \end{align*}
    Observe that a multiple of $x_2$ is the sum of $x_1$ and $x_3$,
    which are two divisors of $x_2 + 1$.
    In particular, $x_2 y_2 = x_1 + x_3 \leq 2x_2 + 2$,
    so $y_2 \leq 4$. We completely determine the positive integral
    solutions with case-work.
    
    \begin{addmargin}[3em]{0em}
    \hspace{\parindent} \emph{Case $y_2 = 1$.}
    In this case, $x_2 = x_1 + x_3$. Then either $x_1$ or $x_3$
    must be $\frac{x_2 + 1}{2}$ or $\frac{x_2 + 1}{3}$,
    since the other divisors of $x_2 + 1$ are too small (or large) to sum to $x_2$.
    If either $x_1$ or $x_3$ equals $\frac{x_2 + 1}{2}$,
    then the other must be $\frac{x_2 - 1}{2}$;
    these are both divisors of $x_2 + 1$ only if $x_2 = 3$ or $5$.
    These correspond to the $4$ positive integer solutions:
    \begin{multicols}{2}
      \begin{itemize}
        \small
        \item $(1, 3, 2, 4, 1, 2)$,
        \item $(2, 3, 1, 2, 1, 4)$,
        \item $(2, 5, 3, 3, 1, 2)$,
        \item $(3, 5, 2, 2, 1, 3)$.
      \end{itemize}
    \end{multicols}
    The other case is if $x_1$ or $x_3$ is $\frac{x_2 + 1}{3}$,
    i.e. $x_2 = 2$.
    This corresponds to the single positive integer solutions:
    \begin{itemize}
      \small
      \item $(1, 2, 1, 3, 1, 3)$.
    \end{itemize}

    \emph{Case $y_2 = 2$.}
    In this case, $2x_2 = x_1 + x_3$. Then either $x_1$ or $x_3$ must be $x_2 + 1$
    or $x_2$ since the other divisors of $x_2 + 1$ are too small to sum to $2x_2$.
    If $x_2 + 1$ is either $x_1$ or $x_3$, then the other must be $x_2 - 1$.
    But $x_2 - 1$ is a positive integral divisor of $x_2 + 1$ only if $x_2 = 2$ or $3$.
    These correspond to the $4$ positive integer solutions:
    \begin{multicols}{2}
      \begin{itemize}
        \small
        \item $(1, 2, 3, 3, 2, 1)$,
        \item $(2, 3, 4, 2, 2, 1)$,
        \item $(3, 2, 1, 1, 2, 3)$,
        \item $(4, 3, 2, 1, 2, 2)$.
      \end{itemize}
    \end{multicols}
    The other case is if $x_2$ divides $x_2 + 1$, i.e. $x_2 = 1$.
    This corresponds to the single positive integer solutions:
    \begin{itemize}
      \small
      \item $(1, 1, 1, 2, 2, 2)$.
    \end{itemize}
    
    \emph{Case $y_2 = 3$.}
    Since $3x_2 = x_1 + x_3 \leq 2x_2 + 2$,
    we necessarily have that $x_2 \leq 2$.
    In particular, $(x_1, x_3)$ must be $(1, 2)$, $(2, 1)$, or $(3, 3)$.
    This corresponds to the $3$ positive integer solutions:
    \begin{multicols}{2}
      \begin{itemize}
        \small
        \item $(1, 1, 2, 2, 3, 1)$,
        \item $(2, 1, 1, 1, 3, 2)$,
        \item $(3, 2, 3, 1, 3, 1)$.
      \end{itemize}
    \end{multicols}
    
    \emph{Case $y_2 = 4$.}
    Since $4 x_2 = x_1 + x_3 \leq 2x_2 + 2$,
    we necessarily have that $x_2 = 1$, $x_1 = 2$, $x_3 = 2$.
    This corresponds to the single positive integer solution:
    \begin{itemize}
      \small
      \item $(2, 1, 2, 1, 4, 1)$.
    \end{itemize}
    \end{addmargin}

    Altogether we have the $14$ positive integral solutions
    $(x_1, x_2, x_3; y_1, y_2, y_3) \in Y_1(A_3)(\N)$
    listed.
\end{proof}


\subsection{Friezes of type \texorpdfstring{$B_2$}{B2}}

Fontaine--Plamondon \cite[Theorems 4.2 and 4.3]{fontaine-plamondon}
showed that there are precisely $6$ positive integral friezes of type $B_2 = C_2$.
By Proposition \ref{prop:correspondence}, there are
precisely $6$ positive integral points on $Y_1(B_2)$.

We give another proof
by directly counting the points of $Y_1(B_2)(\N)$.

\begin{proposition}
  \label{prop:B2}
  The affine surface $Y_1(B_2)$ has exactly
  $6$ positive integral points:
  \begin{multicols}{3}
    \begin{itemize}
      \small
      \item $(1, 1; 2, 2)$,
      \item $(1, 2; 3, 1)$,
      \item $(2, 1; 1, 5)$,
      \item $(2, 5; 3, 1)$,
      \item $(3, 2; 1, 5)$,
      \item $(3, 5; 2, 2)$.
    \end{itemize}
  \end{multicols}
\end{proposition}

\begin{proof}
    Observe that $B_2 = C_2$.
    We use the $B_2$-frieze equations of \nameref{sec:table-A}:
    \begin{align*}
      x_1 y_1 &= x_2 + 1, \\
      x_2 y_2 &= x_1^2 + 1.
    \end{align*}
    
    Observe that $Y_1(B_2)(\N)$
    is in bijection with the positive integral points on
    the affine cubic surface,
    \begin{align*}
      x y z &= x^2 + y + 1,
    \end{align*}
    under the change of variables
    \begin{align*}
      x_1 \longmapsto x, \qquad
      x_2 \longmapsto \frac{x^2 + 1}{y}, \qquad
      y_1 \longmapsto z, \qquad
      y_2 \longmapsto y.
    \end{align*}
    
    Observe that $x_2$ is congruent to $-1 \pmod {x_1}$
    from the $B_2$-frieze equations
    and that $y_2$ is also congruent to $-1 \pmod {x_1}$
    from the affine cubic equation.
    Writing $x_2 = mx_1 - 1$ and $y_2 = nx_1 - 1$,
    we have that $(mx_1 - 1) (nx_1 - 1) = x_1^2 + 1$
    so $(1 - mn)x_1^2 + (m + n)x_1 = 0$.
    Since $x_1 \neq 0$, we have reduced the problem
    to solving $x_1 = \frac{m + n}{mn - 1}$ over the positive integers.
    This is only possible
    for $m, n \in \set{1, 2, 3}$.
    These correspond to the $6$ positive integral solutions
    $(x_1, x_2; y_1, y_2) \in Y_1(B_2)(\N)$
    that are listed.
\end{proof}


\subsection{Friezes of type \texorpdfstring{$G_2$}{G2}}

Fontaine--Plamondon \cite[Theorem 4.4]{fontaine-plamondon}
showed that there are precisely $9$ positive integral friezes of type $G_2$ 
by using the enumeration of $D_4$-friezes by Morier-Genoud--Ovsienko--Tabachnikov
\cite{morier-genoud-ovsienko-tabachnikov}
and the observation that $G_2$-friezes
lift to certain endomorphism-invariant $D_4$-friezes.
By Proposition \ref{prop:correspondence}, there are
precisely $9$ positive integral points on $Y_1(G_2)$.

We give another proof
by directly counting the points of $Y_1(G_2)(\N)$.

\begin{proposition}
  \label{prop:G2}
  The affine surface $Y_1(G_2)$ has exactly
  $9$ positive integral points:
  \begin{multicols}{3}
    \begin{itemize}
      \small
      \item $(1, 1, 2, 2)$,
      \item $(1, 2, 3, 1)$,
      \item $(2, 1, 2, 9)$,
      \item $(2, 3, 2, 3)$,
      \item $(2, 9, 5, 1)$,
      \item $(3, 2, 1, 14)$,
      \item $(3, 14, 5, 2)$,
      \item $(5, 9, 2, 14)$,
      \item $(5, 14, 3, 9)$.
    \end{itemize}
  \end{multicols}
\end{proposition}

\begin{proof}
    The $G_2$-frieze equations are:
    \begin{align*}
    x_1 y_1 &= x_2 + 1, \\
    x_2 y_2 &= x_1^3 + 1.
    \end{align*}
    
    Observe that $Y_1(G_2)(\N)$
    is in bijection with the
    positive integer points on
    the affine cubic surface,
    \begin{align*}
    x y z &= x^3 + y + 1,
    \end{align*}
    under the change of variables
    \begin{align*}
    x_1 \longmapsto x, \qquad
    x_2 \longmapsto \frac{x^3 + 1}{y}, \qquad
    y_1 \longmapsto z, \qquad
    y_2 \longmapsto y.
    \end{align*}
    
    Mohanty \cite[Theorem 2]{mohanty-1977} proved that
    this affine cubic surface has exactly $9$
    positive integral points;
    we now give a shorter proof that
    it only has $9$ positive integral points. 
    Observe that $x_2$ is congruent to $-1 \pmod {x_1}$
    from the $G_2$-frieze equations
    and that $y = y_2$ is also congruent to $-1 \pmod {x_1}$
    from the affine cubic equation.
    Writing $x_2 = mx_1 - 1$ and $y_2 = nx_1 - 1$,
    we have that $(mx_1 - 1) (nx_1 - 1) = x_1^3 + 1$
    so $x_1^3 - mnx_1^2 + (m + n)x_1 = 0$.
    Since $x_1 \neq 0$, we have reduced the problem
    to the quadratic equation
    \[
      x_1^2 - mn x_1 + (m + n) = 0.
    \]
    If the quadratic polynomial has a positive integral root,
    then it has positive integral roots $a$ and $b$ such that
    $a + b = mn$ and $ab = m + n$. This is only possible
    for $a, b, m, n \in \set{1, 2, 3, 5}$.
    These correspond to the $9$ positive integral solutions
    $(x_1, x_2; y_1, y_2) \in Y_1(G_2)(\N)$
    that are listed.
\end{proof}

\begin{remark}
  \label{rem:G2}
    From a na\"{i}ve arithmetic geometry perspective,
    it is not obvious \textit{a priori} that
    the surface $x y z = x^3 + y + 1$ should have
    only finitely many positive integral points.
    Like in the $A_2$ case,
    there are infinite families of integer points
    $(t, -t - 1, -t + 1)$, $(-1, t, -1)$, and $(t, -t^3 - 1, 0)$
    on this affine cubic. Furthermore, every choice of
    positive integer $z$ specializes the surface
    to an affine cubic curve of geometric genus $0$
    with two complex points at infinity; Siegel's theorem
    does not suggest that each curve should have
    only finitely many integral points.
    But in fact, a descent argument can demonstrate
    that all integral points of the surface
    are obtained from the three curves listed above
    and the automorphism $T: (x,y,z) \mapsto (q,y,w)$
    where $q = \frac{y+1}{x} = yz-x^2$ and $w = \frac{y+xz+1}{x^2}$.
    A forthcoming work of Corvaja--Zannier \cite{corvaja-zannier-new}
    provides a more in-depth discussion on the
    integral points of this surface
    and other related examples with interesting topological considerations.
\end{remark}


\newpage
\section*{Table 1}
\label{sec:table-A}
For each finite Dynkin type $\Delta_n$,
we list the Cartan matrix used throughout the paper,
the corresponding system of equations
for $Y_1(\Delta_n)$ given by Definition \ref{def:frieze-polynomial},
the frieze periods,
and the number of positive integer solutions
(equivalently, the number of positive integral friezes).
Let $d(m)$ be the number of divisors of an integer $m$.

\begin{center}
  \resizebox{1.05\textwidth}{!}{
  \begin{tabular}{|c | c | c | c | c|} 
    \hline
    Finite type & Cartan matrix & Frieze equations for & Period & Positive integral point count \\
    $\Delta_n$ & $C_{\Delta_n}$ & $Y_1(\Delta_n)$ & $P_{\Delta_n}$ & $\#Y_1(\Delta_n)(\N)$ \\ [0.5ex] 
    \hline\hline
    $A_n$
      & {\tiny$
          \setlength{\arraycolsep}{3pt}
          \renewcommand{\arraystretch}{0.9}
          \paren{
          \begin{array}{ccccc}
           2 & -1 & & & \\
          -1 & 2 & \ddots & & \\
            & \ddots & \ddots & -1 & \\
            & & -1 & 2 & -1 \\
            & & & -1 & 2
         \end{array}}$}
      & {\tiny $\!\begin{aligned}
        x_1 y_{1} &= x_2 + 1 \\
        x_2 y_{2} &= x_1 + x_3 \\
          &\enspace \vdots \\
        x_{n-1} y_{n-1} &= x_{n-2} + x_n \\
        x_n y_{n} &= x_{n-1} + 1
      \end{aligned}$}
      & $n + 3$
      & $\displaystyle \frac{1}{n+2}{\binom{2n+2}{n+1}}$ \\ 
    \hline
    $B_n$
      & {\tiny$
          \setlength{\arraycolsep}{3pt}
          \renewcommand{\arraystretch}{0.9}
          \paren{
          \begin{array}{ccccc}
           2 & -2 & & & \\
          -1 & 2 & -1 & \phantom{\ddots} & \\
            & -1 & \ddots & \ddots & \\
            & & \ddots & 2 & -1 \\
            & & \phantom{\ddots} & -1 & 2
         \end{array}}$}
      & {\tiny $\!\begin{aligned}
          x_1 y_{1} &= x_2 + 1 \\
          x_2 y_{2} &= x_1^2 + x_3 \\
            &\vdots \\
          x_{n-1} y_{n-1} &= x_{n-2} + x_n \\
          x_n y_{n} &= x_{n-1} + 1
        \end{aligned}$}
      & $n + 1$
      & $\displaystyle \sum_{m=1}^{\lfloor \sqrt{n+1} \rfloor}
        \binom{2n - m^2 + 1}{n}$ \\
    \hline
    $C_n$
      & {\tiny$
          \setlength{\arraycolsep}{3pt}
          \renewcommand{\arraystretch}{0.9}
          \paren{
          \begin{array}{ccccc}
           2 & -1 & & & \\
          -2 & 2 & -1 & \phantom{\ddots} & \\
            & -1 & \ddots & \ddots & \\
            & & \ddots & 2 & -1 \\
            & & \phantom{\ddots} & -1 & 2
         \end{array}}$}
      & {\tiny $\!\begin{aligned}
          x_1 y_{1} &= x_2^2 + 1 \\
          x_2 y_{2} &= x_1 + x_3 \\
            &\vdots \\
          x_{n-1} y_{n-1} &= x_{n-2} + x_n \\
          x_n y_{n} &= x_{n-1} + 1
        \end{aligned}$}
      & $n + 1$
      & $\displaystyle \binom{2n}{n}$ \\
    \hline
    $D_n$
      & {\tiny$
          \setlength{\arraycolsep}{3pt}
          \renewcommand{\arraystretch}{0.2}
          \paren{
          \begin{array}{ccccccc}
           2 & & -1 & \phantom{\ddots} & & \\
             & 2 & -1 & \phantom{\ddots} & & \\
          -1 & -1 & 2 & -1 & \phantom{\ddots} \\
            & & -1 & \ddots & \ddots & \\
            & & & \ddots & 2 & -1 \\
            & & & \phantom{\ddots} & -1 & 2 & \\
         \end{array}}$}
      & {\tiny $\!\begin{aligned}
              x_1 y_{1} &= x_3 + 1 \\
              x_2 y_{2} &= x_3 + 1 \\
              x_3 y_3 &= x_1 x_2 + x_4 \\
              x_4 y_4 &= x_3 + x_5 \\
                &\vdots \\
              x_{n-1} y_{n-1} &= x_{n-2} + x_n \\
              x_n y_{n} &= x_{n-1} + 1
            \end{aligned}$}
      & $\begin{cases}
          n & \text{if $n$ is even}, \\
          2n & \text{if $n$ is odd}
        \end{cases}$
      & $\displaystyle \sum_{m=1}^{n} d(m) \binom{2n - m - 1}{n - m}$ \\

    \hline
    $E_6$
      & {\tiny$
          \setlength{\arraycolsep}{3pt}
          \renewcommand{\arraystretch}{0.9}
          \paren{
          \begin{array}{cccccc}
           2 & & & -1 & & \\
           & 2 & -1 & & & \\
           & -1& 2& -1 & &  \\
           -1 & & -1 & 2 & -1 & \\
           & & & -1 & 2 & -1 \\
           & & & & -1 & 2
         \end{array}}$}
      & {\tiny $\!\begin{aligned}
          x_1 y_1 &= x_4 + 1 \\
          x_2 y_2 &= x_3 + 1 \\
          x_3 y_3 &= x_2 + x_4 \\
          x_4 y_4 &= x_1 x_3 + x_5 \\
          x_5 y_5 &= x_4 + x_6 \\
          x_6 y_6 &= x_5 + 1
        \end{aligned}$}
      & $14$
      & $868$ \\
    \hline
    $E_7$
      & {\tiny$
          \setlength{\arraycolsep}{3pt}
          \renewcommand{\arraystretch}{0.9}
          \paren{
          \begin{array}{ccccccc}
           2 & & & -1 & & & \\
           & 2 & -1 & & & & \\
           & -1& 2& -1 & & & \\
           -1 & & -1 & 2 & -1 & & \\
           & & & -1 & 2 & -1 & \\
           & & & & -1 & 2 & -1 \\
           & & & & & -1 & 2 \\
         \end{array}}$}
      & {\tiny $\!\begin{aligned}
          x_1 y_1 &= x_4 + 1 \\
          x_2 y_2 &= x_3 + 1 \\
          x_3 y_3 &= x_2 + x_4 \\
          x_4 y_4 &= x_1 x_3 + x_5 \\
          x_5 y_5 &= x_4 + x_6 \\
          x_6 y_6 &= x_5 + x_7 \\
          x_7 y_7 &= x_6 + 1
        \end{aligned}$}
      & $10$
      & $4400$ \\
    \hline
    $E_8$
      & {\tiny$
          \setlength{\arraycolsep}{3pt}
          \renewcommand{\arraystretch}{0.9}
          \paren{
          \begin{array}{cccccccc}
           2 & & & -1 & & & & \\
           & 2 & -1 & & & & & \\
           & -1& 2& -1 & & & & \\
           -1 & & -1 & 2 & -1 & & & \\
           & & & -1 & 2 & -1 & & \\
           & & & & -1 & 2 & -1 & \\
           & & & & & -1 & 2 & -1 \\
           & & & & & & -1 & 2 \\
         \end{array}}$}
      & {\tiny $\!\begin{aligned}
          x_1 y_1 &= x_4 + 1 \\
          x_2 y_2 &= x_3 + 1 \\
          x_3 y_3 &= x_2 + x_4 \\
          x_4 y_4 &= x_1 x_3 + x_5 \\
          x_5 y_5 &= x_4 + x_6 \\
          x_6 y_6 &= x_5 + x_7 \\
          x_7 y_7 &= x_6 + x_8 \\
          x_8 y_8 &= x_7 + 1
        \end{aligned}$}
      & $16$
      & $26952$ \\
    \hline
    $F_4$
      & {\tiny$\begin{pmatrix}
        2 & -1 & & \\
        -1 & 2 & -2 & \\
          & -1 & 2 & -1\\
          & & -1 & 2
      \end{pmatrix}$}
      & {\tiny $\!\begin{aligned}
               x_1 y_1 &= x_2 + 1 \\
               x_2 y_2 &= x_1 + x_3 \\
               x_3 y_3 &= x_2^2 + x_4 \\
               x_4 y_4 &= x_3 + 1
             \end{aligned}$}
      & $7$
      & $112$ \\
    \hline
    $G_2$
      & {\tiny$\begin{pmatrix}
           2 & -3 \\
           -1 & 2
         \end{pmatrix}$}
      & {\tiny $\!\begin{aligned}
          x_1 y_1 &= x_2 + 1 \\
          x_2 y_2 &= x_1^3 + 1
        \end{aligned}$}
      & $4$
      & $9$ \\
    [1ex] 
    \hline
  \end{tabular}}
\end{center}


\bibliography{bibliography}{}
\bibliographystyle{alpha}

\end{document}